\documentclass [openany, oneside,11pt,
a4paper]  {article}

\usepackage{dsfont}
\usepackage{amsmath}
\numberwithin{equation}{section}
\usepackage{amssymb}	
\usepackage{bbm}
\usepackage[round]{natbib}
\usepackage[a4paper,
scale =0.75]{geometry}
\usepackage[onehalfspacing]{setspace}
\usepackage{booktabs}
\usepackage{caption}
\usepackage{graphicx}
\usepackage{widetable}
\usepackage{tabularx}

\usepackage{amsthm} 
                          	
\newtheorem{counter1}{not to be used as environment} [section]
	\newtheorem{The}{Theorem}
	\newtheorem*{The*}{Theorem A}
	
	\newtheorem{Lem}[counter1]{Lemma}
	\newtheorem{Kor}[counter1]{Corollary}
	
	\theoremstyle{definition}
	
	\newtheorem{Def}[counter1]{Definition}
	
	\newtheorem{Bei}[counter1]{Example}

	\newtheorem{Bem}[counter1]{Remark}

\begin{document}

\title{Change-point tests under local alternatives for long-range dependent processes}
\author{Johannes Tewes\footnote{Fakult\"{a}t f\"{u}r Mathematik, Ruhr-Universit\"{a}t Bochum, 44780 Bochum, Germany, Email address: Johannes.Tewes@rub.de} }

\maketitle

\begin{abstract}
We consider the change-point problem for the marginal distribution of subordinated Gaussian processes that exhibit long-range dependence. The asymptotic distributions of Kolmogorov-Smirnov- and Cram\'{e}r-von Mises type statistics are investigated under local alternatives. By doing so we are able to compute the asymptotic relative efficiency of the mentioned tests and the CUSUM test. In the special case of a mean-shift in Gaussian data it is always $1$. Moreover our theory covers the scenario where the Hermite rank of the underlying process changes. 

In a small simulation study we show that the theoretical findings carry over to the finite sample performance of the tests..  
\end{abstract}

\section{Introduction}

Over the last two decades various authors have studied the change-point problem under long-range dependence and classical methods are often found to yield different results than under short-range dependence. The CUSUM test is studied in \citet{CsHo} and compared to the Wilcoxon change-point test  in \citet{DeRoTa}. \citet{Lin} investigates a Darling-Erd\H{o}s-type result for a parametric change-point test, and estimators for the time of change are considered in \citet{HoKo} and \citet{NiHaWyZh}. Moreover, the special features of long memory motivated new procedures. \citet{BeTe} and \citet{HoSh2} are testing for a change in the linear dependence structure of the time series and \citet{BeHoKoSh} and \citet{BaPi} construct tests in order to discriminate between stationary long memory observations and short memory sequences with a structural change. For a general overview of the change-point problem under long-range dependence see \citet{KoLe} and the associated chapter in \citet{BeFeGhKu}.

One of the classical change-point problems is the change of the marginal distributions of a time series $\{Y_i\}_{i \geq 1}$. When testing for at most one change-point (AMOC) in the marginal distribution one often considers the empirical distribution function of the first $k$ observations and that of the remaining observations. Taking a distance between the empirical distributions and the maximum over all $k<n$ yields a natural statistic. Common distances are the supremum norm, which gives the Kolmogorov-Smirnov statistic
\begin{align} 
T_n =  \max_{1 \leq k < n} \sup_{x \in \mathds R} \left \lvert \sum_{i=1}^k 1_{ \{ Y_{i} \leq x \} } - \frac{k}{n} \sum_{i=1}^n 1_{ \{ Y_{i} \leq x \} } \right \rvert,  \label{EinleitungKSStatistik}
\end{align}
or an $L^2$-distance, which gives the Cram\'{e}r-von Mises statistic
\begin{align} 
S_n =   \max_{1 \leq k < n}  \int_{x \in \mathds R} \left ( \sum_{i=1}^k 1_{ \{ Y_{i} \leq x \} } - \frac{k}{n} \sum_{i=1}^n 1_{ \{ Y_{i} \leq x \} } \right)^2 \ d \hat F_n(x). \label{EinleitungCvMStatistik}
\end{align}
Both are widely used for goodness-of-fit tests and two-sample problems. In the change-point literature they are considered by \citet{Szy} for independent data, by \citet{Ino} for strongly mixing sequences and by \citet{GiLeSu} for linear long-memory processes. However, note that in the LRD setting only the Kolmogorov-Smirnov test has been investigated. 

(\ref{EinleitungKSStatistik}) and (\ref{EinleitungCvMStatistik}) are functionals of the sequential empirical process, that is $\sum_{i=1}^{\lfloor nt \rfloor} (1_{\{ Y_i \leq x \}} - F(x))$ for $t \in [0,1]$ and $x \in \mathds R$. Thus the asymptotic distributions of $T_n$ and $S_n$ rely on that of the sequential empirical process. For weakly dependent sequences this would be a Gaussian process, in the special case of independent random variables it is called Kiefer-M\"{u}ller process. For stationary sequences that exhibit long-range dependence, \citet{DeTa} proved that the limit process is of the form $\{J(x)Z(t)\}_{t,x}$, where $J(x)$ is a deterministic function and the process is therefore called semi-degenerate. They considered subordinated Gaussian processes, in detail $Y_i = G(X_i)$ for any measurable function $G$ and a Gaussian sequence $X_i$ with non-summable autocovariance function. A similar limit structure was later obtained independently by \citet{HoHs} and \citet{GiKoSu} for long-range dependent moving-average sequences.

It is the main goal of this paper to derive the limit distribution of change-point statistics of the type (\ref{EinleitungKSStatistik}) and (\ref{EinleitungCvMStatistik}) under local alternatives. We then apply these results to derive the asymptotic relative efficiency (ARE) of several change-point tests. To this end we investigate the sequence
\begin{align}
G_1(X_1), \dots , G_1(X_{k^{\ast}}), G_n(X_{k^{\ast}+1},) \dots G_n(X_n), \label{IntorG}
\end{align}
Here $G_n$ is a sequence of functions such that the distribution of $G_n(X_1)$ converges to the distribution of $G(X_1)$ in some suitable way.  

Therefore, we are able to analyze various types of change-points, among them a mean-shift. Thus we may compute the ARE of Kolmogorov Smirnov, Cram\'{e}r-von Mises, CUSUM and Wilcoxon test and get the surprising result that in case of Gaussian data it is always $1$. \\
The mathematically most challenging case is the situation when the Hermite rank changes. The Hermite rank of the class $\{1_{ \{ G(\cdot) \leq x \}  } - F(x)\}_{x \in \mathds R}$ is defined as the smallest positive integer, such that $E [1_{\{ G(X_1) \leq x \} } H_q(X_1)] \not =0$ for some $x \in \mathds R$, with $H_q$ being the $q$-th Hermite polynomial. The structure of the limiting process $Z(t)$, e.g. the marginal distribution and the covariance structure, mainly depends on $m$. However, a special feature of distributional changes in subordinated Gaussian processes is the fact that the Hermite rank may change, too. Hence the question arises which Hermite process will determine the limit distribution. Under a mean-shift the Hermite rank remains unchanged, which can be seen easily by its definition. 

Our results differ in various ways from those obtained in \citet{GiLeSu}, where changes in the coefficients of an LRD linear process were investigated. While the empirical process of LDR moving average sequences converges to fractional Brownian motion, we may encounter higher order Hermite processes. The possible change in the Hermite rank is therefore a novel feature in our investigation. 

The rest of the paper is organized as follows. In section \ref{EmpProcessChangePointAlt} we will state a limit theorem for the sequential empirical process under change-point alternatives. Moreover we will give the asymptotic distribution of the test statistics under the hypothesis of no change as well as under local alternatives. Thus we are able to derive the asymptotic relative efficiency of several change-point tests. In section \ref{LimitTheoremsArrays} we consider the empirical process for long-range dependent arrays that are stationary within rows. The outcome mainly serves as a device for proving the main results, but is also of interest on its own. Section \ref{SimStudy} contains the simulation study. To the best of our knowledge there are no results on the finite sample performance of the Cram\'{e}r-von Mises change-point test under long memory. It is compared to other change-point tests and the effect of an estimated Hurst-coefficient is discussed. We obtain that the theoretical results (e.g. asymptotic relativ efficiency between Cram\'{e}r-von Mises and CUSUM test) carry over to the finite sample performance of the tests. Finally proofs are provided in section \ref{SectionProofs}.

\section{Main results} \label{EmpProcessChangePointAlt}

Let $ \{X_i\}_{i \geq 1}$ be a stationary Gaussian process, with
\begin{align*}
E X_i =0, \ \  EX_i^2 =1 \ \ \text{and} \ \ \rho (k) = EX_0 X_k = k^{-D} L(k),
\end{align*}
for $0 < D <1$ and a slowly varying function $L$. The non-summability of the covariance function is one possibility to define long-range dependence. We investigate our results for so called subordinated Gaussian processes $\{Y_i \}_{i \geq 1}$, where $Y_i = G(X_i)$ and $G \colon \mathds R \to \mathds R$ is a measurable function. The key tool in our analysis of possible changes in the marginal distribution of such a process is the sequential empirical process. To obtain weak convergence of this process the right normalization is given by $d_{n,m}$, defined by
\begin{align}
d_{n,m}^2 = Var \left ( \sum_{i=1}^n H_m(X_i) \right )  \sim n^{2H} L^{m}(n), \label{DefinitionNormalizationBlau}
\end{align}
where the constant of proportionality is $2m!(1-mD)^{-1}(2-mD)^{-1}$, see Theorem 3.1 in \citet{Taq}. $H = 1-mD/2$ is called Hurst coefficient and 
\begin{align*}
m =  \min \left \{ q >0  \ \vert  \ E [1_{\{ G(X_1) \leq x \} } H_q(X_1)] \not =0 \text{ for some } x \right \},
\end{align*}
is the Hermite rank of $\{1_{ \{ G(\cdot) \leq x \}  } - F(x)\}_{x \in \mathds R}$. The mentioned result of \citet{DeTa} then reads as follows.

\begin{The*} [Dehling, Taqqu] \label{SeqEmpProHypo}
Let the class of functions $\{1_{ \{ G(\cdot) \leq x \}  } - F(x)\}_{x \in \mathds R}$ have Hermite rank m and let $0 < D < 1/m$. Then
\begin{align}
\frac 1 {d_{n,m}} \sum_{i=1}^{\lfloor nt \rfloor} (1_{\{ G(X_i) \leq x \}} - F(x)) \xrightarrow{\mathcal D} \frac {J_m(x)} {m !} Z_{m,H}(t) \label{SeqKonvHypo}
\end{align}
where the convergence takes place in $D([0,1] \times [- \infty,\infty])$, equipped with the uniform topology.  $J_m(x)$ is defined by
\begin{align*}
J_m(x) = E[1_{\{ G(X_1) \leq x \} } H_m(X_1)] 
\end{align*}
 and $(Z_{m,H}(t))_{t \in [0,1]}$ is an $m$-th order Hermite process, see \citet{Taq3} for a definition.
\end{The*}

\begin{Bem} In the case $m=1$, the Hermite process becomes the well known fractional Brownian Motion, which we denote by $B_H(t)$.
\end{Bem}

\subsection{The empirical process under change-point alternatives}

Let us consider the following change-point model. Define the triangular array
\begin{align}
Y_{n,i} =  \begin{cases} 
                             G(X_i),  & \text{if }i \leq \lfloor n \tau \rfloor ,\\ 
                          G_n(X_i), & \text{if } i \geq \lfloor n \tau \rfloor +1,
                                                      \end{cases} \label{ArraymitChange} 
\end{align}
for measurable functions $G$ and $(G_n)_n$ and unknown $\tau \in (0,1)$. For $\tau =0$ one gets a row-wise stationary triangular array, as considered in section \ref{LimitTheoremsArrays}, and for $\tau =1$ a stationary sequence, as in \citet{DeTa}. In what follows we will denote the distribution functions of $G(X_i)$ and $G_n(X_i)$ by $F$ and $F_{(n)}$, respectively. \\
To obtain weak convergence of the empirical process of (\ref{ArraymitChange}) we have to make some assumptions on the structure of the change and the Hermite rank.
\\\\
\textbf{Assumption A:} 
\begin{enumerate}
\item [A1.]The class of functions $\{1_{ \{ G(\cdot) \leq x \}  }\}_{x \in \mathds R}$ has Hermite rank $m$ with $0 < D < 1/m$.
\item [A2.] Let $m(n)$ be the Hermite rank of $\{1_{ \{ G_n(\cdot) \leq x \}  }\}_{x \in \mathds R}$ and $m^{\ast}= \liminf_{n \to \infty} m(n)$. Then we assume
\begin{align*}
n^{(m-m^{\ast})D(1+\delta)/2} \sup_{x \in \mathds R} ( P ( \min\{G(X_1), G_n(X_1) \} \leq x ) -P ( \max\{G(X_1), G_n(X_1) \} \leq x )) \rightarrow 0,
\end{align*}
 for some $\delta >0$.

\end{enumerate}

\begin{The} \label{ConvChangeAltA}
If Assumption A holds, then
\begin{align*}
\frac 1 {d_{n,m}} \sum_{i=1}^{\lfloor nt \rfloor} (1_{\{ Y_{n,i} \leq x \}} - P(Y_{n,i} \leq x)) \xrightarrow{\mathcal D} \frac {J_m(x)} {m!} Z_{m,H}(t),
\end{align*}
where $J_m(x)$ is the Hermite coefficient of $1_{\{G() \leq x\}}$. The convergence takes place in $D([0,1] \times [- \infty,\infty])$, equipped with the uniform topology.
\end{The}

\begin{Bem} \label{ErklaerungAss3}
(i) For given functions $G(x)$ and $G_n(x)$, Assumption A2 might easily being checked, see the examples below. It serves to ensure convergence of the Hermite coefficients $J_{q,n}(x) = E[1_{\{G_n(X_i) \leq x \}}H_q(X_i)]$. In detail,
\begin{align*}
\sup_{x \in \mathds R} ( P ( \min\{G(X_1), G_n(X_1) \} \leq x ) -P ( \max\{G(X_1), G_n(X_1) \} \leq x )) \rightarrow 0
\end{align*}
implies, see the proof of Lemma \ref{KonvergenzHermiteKoeffizientFunktionen},
\begin{align}
\sup_{x \in \mathds R} \lvert J_{q,n}(x) - J_q(x) \rvert \rightarrow 0 \ \ \ \ \forall q \in \mathds N. \label{KonvergenzHermiteFunktionRemark}
\end{align}
By Assumption A1, $J_1(x) = \dots J_{m-1}(x) =0$ for all $x \in \mathds R$, yet $J_m(x) \not =0$ for some $x$. Together with (\ref{KonvergenzHermiteFunktionRemark}) this implies  $m^{\ast}= \liminf_{n \to \infty} m(n) \leq m$.

(ii) Moreover, A2 implies convergence of the marginal distribution function. To see this, note
\begin{align*}
\lvert F_{(n)}(x) - F(x) \rvert =& \ \max \{F_{(n)}(x), F(x) \} -  \min  \{F_{(n)}(x), F(x) \} \\
\leq & \ P (\min \{G(X),G_n(X) \} \leq x) -P(\max \{ G(X),G_n(X) \} \leq x )
\end{align*}
and $n^{(m-m^{\ast})D(1+\delta)/2} = O(1)$. However, the converse is not always true. Consider for instance the functions $G(x)=x$ and $G_n(x) = G_1(x) = -x$ or the situation in Example \ref{BeispielKomplizierterLimes}. Then again, there are lots of natural choices of $G$ and $G_n$ for whom convergence of the marginal distribution functions (with a certain rate) implies Assumption A2. Among them $G_n(x) = G(x) + \mu_n$ (mean-shift), $G_n(x) = \sigma_n G(x)$ (change in variance) and 
\begin{align*}
G_n(x) = F_{(n)}^{-1} \circ \Phi (x) \ \ \ \ \text{and} \ \ \ \ G(x) = F^{-1} \circ \Phi (x).
\end{align*}
(iii) Our assumptions explicitly allow for the Hermite rank to change together with the marginal distribution. Then again, the limit behaivior seems to be untouched by this change. Intuitively this corresponds to the idea that the change in distribution and the change in the Hermite coefficient, both caused by the difference of $G$ and $G_n$, are of the same order. For $q<m$ this enforces the function $J_{q,n}(x)$ to converge rather fast to $0$.
Technically this can be explained through A2. If this assumption is dropped, we might actually encounter limits with multiple Hermite processes. Such cases will be considered in Example \ref{BeispielKomplizierterLimes} and Corollary \ref{KonvEmpProArray}.

(iv) If A1 is violated, the sequence $\{G(X_i)\}_{i\geq 1}$ is actually short-range dependent. For stationary observations \citet{CsMi} showed convergence of the sequential empirical process to a two-parameter Gaussian process. Change-point alternatives have not been considered for such random variables, yet, but would require fundamentally different proofs compared to our results.
\end{Bem}

\subsection{Asymptotic behavior of the change-point statistics}

We now apply the results concerning empirical processes to determine the asymptotic distribution of the Kolmogorov-Smirnov statistics
\begin{align}
T_{n} = \sup_{t \in [0,1]} \sup_{x \in \mathds R} d_{n,m}^{-1} \left \lvert \sum_{i=1}^{ \lfloor nt \rfloor } 1_{ \{ Y_{n,i} \leq x \} } - \frac{\lfloor nt \rfloor}{n} \sum_{i=1}^n 1_{ \{ Y_{n,i} \leq x \} } \right \rvert  \label{NonWeightedTeststatistic}
\end{align}
and that of the Cram\'{e}r-von Mises change-point statistic
\begin{align}
S_n =d_{n,m}^{-2} \sup_{t \in [0,1]} \int_{\mathds R} \left \lvert \sum_{i=1}^{ \lfloor nt \rfloor } 1_{ \{ Y_{n,i} \leq x \} } - \frac{\lfloor nt \rfloor}{n} \sum_{i=1}^n 1_{ \{ Y_{n,i} \leq x \} } \right \rvert ^2 \ d \hat F_n(x). \label{CvMTeststatistic}
\end{align}
To get a non degenerate limit under a sequence of local alternatives it is important to choose the right amount of change. For a mean-shift this is naturally the difference of the expectations before and after the change. For a general change we formulate the test problem as follows: We wish to test the hypothesis
\begin{align*}
\mathbf{H}: \ \ \text{Assumption A1 holds and } G_n(x) = G(x) \text{ for all } x \in \mathds R \text{ and } n \geq 1,
\end{align*}
against the sequence of local alternatives
\begin{align}
\nonumber \mathbf{A_n}: \ \ & \text{Assumption A holds and, for } n \to \infty, \\
& \frac{n}{d_{n,m}}  (F(x) - F_{(n)} (x)) \rightarrow g(x), \label{LocalAlternativeConvergence} \\
\nonumber & \text{uniformly in } x, \text{ where } g(x) \text{ is a measurable function of bounded} \\
\nonumber & \text{total variation, whose support has positive Lebesgue measure}.
\end{align}

\begin{Bem} 
Note that $n d_{n,m}^{-1} \sim n^{mD/2} L^{-m/2}(n)$. Thus (\ref{LocalAlternativeConvergence}) implies
\begin{align*}
n^{(m-m^{\ast})D(1+\delta)/2} (F(x) - F_{(n)} (x)) \rightarrow 0,
\end{align*}
for $\delta < m^{\ast}/ (m-m^{\ast})$ or $m^{\ast} =m$. This again implies Assumption A2 for certain choices of functions  $G$ and $G_n$, see Remark \ref{ErklaerungAss3} (ii).
\end{Bem}

\begin{The} \label{KorAsVertCPStat} (i) Under the hypothesis $\mathbf{H}$ of no change we have, as $n \to \infty$,
\begin{align*}
& \ T_n \xrightarrow{\mathcal{D}} \sup_{x \in \mathds R} \lvert J_m(x)/(m!) \rvert \sup_{t \in [0,1]} \left \lvert \tilde Z_{m,H}(t) \right \rvert \\
\text{and} \ \ & \ S_n \xrightarrow{\mathcal{D}}  \int_{x \in \mathds R} (J_m(x)/(m!))^2 \ d F(x)  \sup_{t \in [0,1]} \left \lvert \tilde Z_{m,H}(t) \right \rvert^2,
\end{align*}
where $\tilde Z_{m,H}(t) =  Z_{m,H}(t) - t  Z_{m,H}(1)$. \\
(ii) Under the sequence of local alternatives $\mathbf{A_n}$ we have, as $ n \to \infty$,
\begin{align*}
& \ T_n \xrightarrow{\mathcal{D}} \sup_{x \in \mathds R} \sup_{t \in [0,1]} \left \lvert J_m(x)/(m!) \tilde Z_{m,H}(t)- g(x)\psi_{\tau}(t) \right \rvert \\ 
\text{and} \ \ & \ S_n \xrightarrow{\mathcal{D}} \sup_{t \in [0,1]}  \int_{x \in \mathds R} \left ( J_m(x)/(m!) \tilde Z_{m,H}(t)- g(x)\psi_{\tau}(t) \right )^2 \ d F(x), 
\end{align*}
where 
\vspace{-10pt}
\begin{align*}
\psi_{\tau} (t) = \begin{cases} 
                             t ( 1- \tau ),  & \text{if }t \leq   \tau  ,\\
                          \tau ( 1-  t ), & \text{if } t > \tau.
                                                      \end{cases}
\end{align*}
\end{The}

Motivated by this Theorem we consider change-point tests based on the statistics $T_n$ and $S_n$. Critical values might be chosen as
 \begin{align*}
\sup_{x \in \mathds R} \lvert J_m(x)/(m!) \rvert q_{1-\alpha,m,H} \ \ \ \ \text{and} \ \ \ \ \int_{x \in \mathds R} (J_m(x)/(m!))^2 \ d F(x) q_{1-\alpha,m,H}^2,
\end{align*}
for the Kolmogorov-Smirnov test and the Cram\'{e}r-von Mises test, respectively. Here $q_{1-\alpha,m,h}$ is the $(1-\alpha)$-quantile of $\sup_{t \in [0,1]} \lvert \tilde Z_{m,H}(t) \rvert$. Thereby the tests have asymptotically level $\alpha$ and nontrivial power against local alternatives. 

The tests can be performed, if the right normalization for the empirical process, the supremum of $J_m(x)$ and the distribution of $ \sup_{ t \in [0,1]} \rvert \tilde Z_{m,H}(t) \rvert$ are known. In practical applications this might be not the case. Solutions are self-normalization (\citet{Sha}), estimating the the Hurst-coefficient (see for example \citet{Kuen}) and bootstrap estimators for $J_m(x)$ (\citet{Tew2}).

\subsection{Examples}

\begin{Bei} [Mean-shift] \label{MeanShift}
Let $G_n(x) = G(x)+ \mu_n$ with $ \mu_n \sim d_n/n$, then we get the typical change in the mean problem. In the case of long-range dependent subordinated Gaussian processes this was considered in \citet{DeRoTa,DeRoTa2}, \citet{CsHo}, \citet{Sha} and \citet{Bet}. Let $f_G$ be the probability density of $G(X_1)$, and assume that it is continuous and of bounded variation. Then we obtain
\begin{align*}
\frac n {d_{n,m}} (F(x) - F_{(n)}(x) ) = \frac n {d_{n,m}} (F(x) - F(x- \mu_n) ) \rightarrow  C f_G(x),
\end{align*}
where, due to continuity of $f_G$, the convergence holds uniformly.
\end{Bei}

\begin{Bei} [Change in the variance] \label{BeispielChangeVarianz}
To describe the change-in-variance-problem define $G_n(x) = 1/(1- \delta_n) G(x)$, with $\delta_n \sim d_n/n$. For ease of notation let $\delta_n = d_n/n$. Then we get
\begin{align}
\nonumber & \ \sup_{x \in \mathds R} \left \lvert \delta_n^{-1} (F(x) - F_{(n)} (x) ) - x f_G(x) \right \rvert \\
\nonumber = & \ \sup_{x \in \mathds R} \left \lvert \delta_n^{-1} (F(x) - F (x- \delta_n x ) ) - x f_G(x) \right \rvert \\
\nonumber = & \ \sup_{x \in \mathds R} \left \lvert \frac  { x F(x) - (x - \delta_n x) F(x- \delta_n x ) } {\delta_n x}  - F(x- \delta_n x ) - x f_G(x) \right \rvert \\
\leq & \ \sup_{x \in \mathds R} \left \lvert \frac  { x F(x) - (x - \delta_n x) F(x- \delta_n x ) } {\delta_n x} - (x f_G(x) + F(x)) \right \rvert \label{changevariance1} \\
& \ +  \sup_{x \in \mathds R} \lvert F(x - \delta_n x) - F(x) \rvert.  \label{changevariance2} 
\end{align}
The derivative of $xF(x)$ is $xf_G(x) + F(x)$, hence (\ref{changevariance1}) converges to $0$. The convergence is uniform, if $f_G$ and $F$ are continuous. (\ref{changevariance2}) converges to $0$, because of continuity, monotonicity and boundedness of $F$. Thus (\ref{LocalAlternativeConvergence}) holds with function $g (x) = x f_G(x).$  Assume without loss of generality $\sigma_n= 1/ (1- \delta_n)>1$, then
\begin{align*}
& \ P\max\{G(X_1), G_n(X_1) \} \leq x ) \\
= & \ P (\sigma_n G(X_1) \leq, G(x) \geq 0) + P( G(X_1) \leq x, G(X_1) \leq 0) \\
= & \ \begin{cases} 
                            F(x/ \sigma_n),  & \text{if } x \geq 0 ,\\
                         F(x), & \text{if } x<0.   
                                                      \end{cases}
\end{align*}
The minimum can be treated analogously, hence Assumption A2 follows from convergence of the marginals.

Additionally one might consider a combined change in mean and variance, given through $G_n(x) =  \sigma_n G(x) + \mu_n$. In this case (\ref{LocalAlternativeConvergence}) holds with $g(x) = f_G(x) (C_1+C_2x)$.
\end{Bei}

\begin{Bei} [Generalized inverse of a mixture distribution] \label{BspMixtureDistribution}
By using the generalized inverse of a distribution function one could generate subordinated Gaussian processes with any given marginals, see for example \citet{DeRoTa2}. We use this for the change-point problem by setting \begin{align*}
& \ G \equiv F^{-1} \circ \Phi \ \ \ \ \text{and} \ \ \ \ G_n \equiv F_{(n)}^{-1} \circ \Phi.
\end{align*}
For a continuous distribution function $F^{\ast}$ define the mixture
\begin{align*}
F_{(n)} (x) = (1- \delta_n) F(x) + \delta_n F^{\ast}(x),
\end{align*}
with $\delta_n \sim d_n n ^{-1}$. Then (\ref{LocalAlternativeConvergence}) holds with $g(x) = F^{\ast}(x) - F(x)$ and moreover
\begin{align*}
P ( \max\{G(X_1), G_n(X_1) \} \leq x )  = & \ P (\max \{F^{-1} \circ \Phi(X_1), F_{(n)}^{-1} \circ \Phi (X_1) \} \leq x ) \\
= & \ P ( \Phi(X_1) \leq \min \{F(x), F_{(n)}(x)\}) \\
= & \ \min \{F(x), F_{(n)}(x)\}.
\end{align*}
Analogously one has $P ( \min\{G(X_1), G_n(X_1) \} \leq x )  = \max \{F(x), F_{(n)}(x)\}$. Hence
\begin{align*}
P ( \min\{G(X_1), G_n(X_1) \} \leq x ) -P ( \max\{G(X_1), G_n(X_1) \} \leq x ) = \lvert F_{(n)}(x) -F(x) \rvert,
\end{align*}
thus Assumption A2 is also satisfied. For strongly mixing data similar local alternatives were considered by \citet{Ino}.
\end{Bei}

\begin{Bei} [$\chi^2$-distribution] \label{BeipsielChiQuadrat}
Consider a $\chi^2$-distribution given through $G(x)=x^2$ and note that the indicator functions have Hermite rank $m=2$, see also \citet{DeTa}. Further let
\begin{align*}
G_n(x) = \begin{cases} 
                            a_n x^2,  & \text{if } x \geq 0 ,\\
                         x^2, & \text{if } x<0,       
                                                      \end{cases}
\end{align*}
with Hermite ranks $m(n)=1$ for all $n \in \mathds N$. If $(a_n-1) \sim d_{n,2}/n$, then one can show (similar to the case of a variance change in Example \ref{BeispielChangeVarianz}) that
\begin{align*}
\frac n {d_{n,2}} \left ( P \left ( G(X_1) \leq x \right ) -P \left ( G_n(X_1) \leq x \right ) \right ) \\
\rightarrow C \sqrt{x} \phi( \sqrt{x} )1_{[0,\infty)} (x),
\end{align*}
uniformly in $x$. As Assumption A2 is satisfied, too, we may apply Corollary \ref{KorAsVertCPStat} (ii) with function $g(x) =C \sqrt{x} \phi( \sqrt{x} )1_{[0,\infty)} (x)$ and $m=2$.
\end{Bei}

\begin{Bei} [Multiple Hermite processes in the limit] \label{BeispielKomplizierterLimes}

In the previous example, together with the marginal distribution, also the Hermite rank has changed. However, the limiting process seems to be untouched by this fact and one might ask whether this is intuitive or not.

It is caused by the fact that the change in the distribution and the change in the Hermite coefficients, both originating in the difference of the functions $G(x)$ and $G_n(x)$, are of the same order. 

To get an additional Hermite process in the limit, one would need $(a_n-1) \sim d_{n,2}/d_{n,1}$, see Corollary \ref{KonvEmpProArray} and its proof. But then 
\begin{align*}
\frac n {d_{n,2}} \sup_x \left \lvert F(x)  - F_{(n)}(x) \right \rvert = \frac n {d_{n,1}} \frac {d_{n,1}} {d_{n,2}} \sup_x \left \lvert F(x)  - F_{(n)}(x) \right \rvert
\rightarrow \infty,
\end{align*}
and the test would have asymptotic power $1$.

To achieve nontrivial asymptotic power one has to consider structural breaks that consists of two aspects and where only one is captured by the marginal distribution. To this end define the transformations
\begin{align*}
G(x) = \Phi ^{-1} (F( \lvert x \rvert ))=  \Phi ^{-1} (2\Phi(\lvert x \rvert )-1)
\end{align*}
and 
\begin{align*}
G_n(x) = \Phi ^{-1} (F_{(n)}^{\ast} (G_n^{\ast} (x) ) + \mu_n,
\end{align*}
where $F_{(n)}^{\ast}(x) = P(G_n^{\ast}(X_i) \leq x)$ and 
\begin{align*}
G_n^{\ast}(x) = \begin{cases} 
                            a_n x^2,  & \text{if } x \geq 0 ,\\
                         x^2, & \text{if } x<0,       
                                                      \end{cases}
\end{align*}
for some sequence $(a_n)_n$ with $a_n \not=1$ and $a_n \rightarrow 1$.
On the one hand, $\{1_{\{G() \leq x \}} \}_x$ has Hermite rank $m=2$ and $G(X_i) \sim N(0,1)$. On the other hand, $\{1_{\{G_n() \leq x \}} \}_x$ has Hermite rank $m(n)=1$ for all $n \in \mathds N$ and $G_n(X_i) \sim N(\mu_n,1)$. Now let $\mu_n \sim d_{n,2}/n$, then Example \ref{MeanShift} applies and we obtain
\begin{align*}
\frac n {d_{n,2} } (F_{(n)}(x) - F(x) ) = \frac n {d_{n,2} } (\Phi(x-\mu_n) - \Phi(x) )  \rightarrow C \phi(x),
\end{align*}
for any sequence $(a_n)_n$. In contrast, the convergence of the Hermite coefficients is highly influenced by $(a_n)_n$. If the sequence is chosen such that $(a_n-1) \sim d_{n,2}/ d_{n,1}$ (therefore, it converges slower than $\mu_n$), then the sequential empirical process will converge towards
\begin{align*}
K(x,t) = \begin{cases} 
                            J_2(x)/2 Z_2(t),  & \text{if } t \leq \tau ,\\
                         \tilde J_1(x) Z_1(t) + J_2(x)/2 Z_2(t) + , & \text{if } t > \tau.       
                                                      \end{cases}
\end{align*}
Actually this can be proved similar to Corollary \ref{KonvEmpProArray}. Moreover, the Kolmogorov-Smirnov statistic converges weakly to
\begin{align*}
\sup_{t \in [0,1]} \sup_{x \in \mathds R} \lvert K(x,t) - tK(x,1) - \psi_{\tau}(t) C \phi(x) \rvert.
\end{align*}
We find this example rather pathological, therefore such situations are excluded from the main results via Assumption A2.
\end{Bei}

\subsection{Asymptotic relative efficiency}

By studying the asymptotic distributions under local alternatives one might compare different tests in terms of the asymptotic relative efficiency (ARE). Here we give a precise definition of the ARE in the very special context of our change-point setting. The general idea is due to \citet{Pit} (for a published article see for example \citet{Noe2}) and was formalized in \citet{Noe}. Of course it can be extended to all kinds of testing procedures.

\begin{Def} \label{DefinitionARE}
Let $T_1$ and $T_2$ represent two change-point test procedures. Consider the local alternatives 
\begin{align*}
& (G, G_{n_k} , \tau ) \text{ and a sample size } (n_k)_k, \\
& (G, \tilde G_{m_k}, \tau ) \text{ and a sample size } (m_k)_k,
\end{align*}
such that $G_{n_k} (x) = \tilde G_{m_k}(x) = G_k(x)$ for all $k \geq 1$ and $x \in \mathds R$.

Let $\beta_1$ be the asymptotic power of the test $T_1$ against the local alternatives given by $(G, G_{n_k} , \tau , (n_k)_k)$ and $\beta_2$ be the asymptotic power of the test $T_2$ against the local alternatives given by $(G, \tilde G_{m_k} , \tau , (m_k)_k)$. If $\beta_1$ equals $\beta_2$, then the \emph{asymptotic relative efficiency} (ARE) of the tests $T_1$ and $T_2$ is defined as
\begin{align*}
ARE(T_1,T_2) = \lim_{k \to \infty } \frac {m_k}{n_k}.
\end{align*}
\end{Def}

\begin{Bei} [Mean-shift in Gaussian data] \label{AREMeanShift}
Consider $G(x)=x$ and $G_n(x) = G(x)+ \mu_n$, in other words a mean-shift in Gaussian data. As for the Hermite coefficient function, we get $J_1(x) = - \phi(x)$, where $\phi$ is the standard normal probability density. Thus, according to Corollary \ref {KorAsVertCPStat}, the test statistic $T_n$ converges towards
\begin{align*}
 \sup_{x \in \mathds R}  \lvert \phi(x) \rvert \sup_{t \in [0,1]} \left \lvert \tilde B_H(t)+ C \psi_{\tau}(t) \right \rvert = (2 \pi)^{-1/2} \sup_{t \in [0,1]} \left \lvert \tilde B_H(t)+ C \psi_{\tau}(t) \right \rvert,
 \end{align*}
 whereas under the Null, that is we have a stationary standard Gaussian sequence, the limit distribution would be
 \begin{align*}
 \sup_{x \in \mathds R}  \lvert \phi(x) \rvert \sup_{t \in [0,1]} \left \lvert \tilde B_H(t) \right \rvert = (2 \pi)^{-1/2} \sup_{t \in [0,1]} \left \lvert \tilde B_H(t) \right \rvert.
 \end{align*}
For the Cram\'{e}r-von Mises statistic we obtain analogously the limit distributions
\begin{align*}
\int \phi^3(x) dx \sup_{t \in [0,1]} \left \lvert \tilde B_H(t)+ C \psi_{\tau}(t) \right \rvert^2 \ \ \text{and} \ \ \int \phi^3(x) dx, \sup_{t \in [0,1]} \left \lvert \tilde B_H(t) \right \rvert^2
\end{align*}
under local alternative and hypothesis, respectively. Hence in this special case the CUSUM test, the Wilcoxon test (see \citet{DeRoTa2} for each), the Kolmogorov-Smirnov test and the Cram\'{e}r-von Mises test all have the same asymptotic power, namely 
\begin{align}
P \left ( \sup_{t \in [0,1]} \lvert \tilde B_H(t) +C \psi_{\tau} (t) \rvert > q_{1-\alpha, H} \right ), \label{AsymptoticPower}
\end{align}
where $q_{1-\alpha,H}$ is the $(1-\alpha)$-quantile of the maximum of a fractional Brownian bridge $\sup_{t \in [0,1]} \lvert \tilde B_H(t) \rvert$. \\
As a direct consequence, one gets that the ARE of the four tests is $1$. This result is quite surprising, keeping in mind that CUSUM and Wilcoxon tests are designed to detect level-shifts, while our tests have power against all kinds of distributional changes. 
\end{Bei}

For non-Gaussian data and change-points beyond a simple mean-shift, the investigation of the ARE is not that straightforward. In fact, little is known about the distribution of 
\begin{align*}
\sup_t \lvert \tilde B_H(t) + f(t) \rvert,
\end{align*}
and even less if higher order Hermite processes are considered. This seems to prevent a precise computation of the ARE in many cases. However, one might derive lower bounds for the efficiency as we do in next example for a combined change in mean and variance. Unlike in the previous example, we will make use of the subtle definition of the ARE.

\begin{Bei}[Combined change in mean and variance] \label{AREMeanVarChange}

Let $G(x) = x$ and $G_k(x) = \sigma_k x+\mu_k$, that is a combined change of mean and variance in Gaussian data. If further $\mu_k k/d_{k,1} \rightarrow C_1>0$ and $(1-1/\sigma_k)k/d_{k,1} \rightarrow C_2>0$, then by example \ref{BeispielChangeVarianz} the empirical bridge-type process converges to (for the sample size $k \to \infty$)
\begin{align*}
\phi(x) \tilde B_H(t) + \phi(x) (C_1+C_2x) \psi_{\tau}(t), \ \ \ \ x \in \mathds R, \ t \in[0,1].
\end{align*}
We now consider slightly modified Cram\'{e}r-von Mises and CUSUM tests, in detail, instead of $[0,1]$ the supremum is taken over $[\kappa_1,\kappa_2]$ for some $\kappa_1 \in (0,1/2)$ and $\kappa_2 \in (1/2,1)$. 

The asymptotic distribution of the CUSUM test has been derived in \citet{DeRoTa2}, but only in the case of a mean-shift with constant variance. However, for $EG^2(X_i) < \infty$ the CUSUM statistic is a continuous functional of the sequential empirical process. Thus, we might apply our Theorem \ref{ConvChangeAltA} and conclude that the CUSUM statistic converges under this type of local alternatives to
\begin{align*}
 & \ \sup_{t \in [\kappa_1,\kappa_2]}  \left \lvert \int \phi(x) \left ( \tilde B_H(t)+ (C_1+C_2 x) \psi_{\tau}(t) \right)  dx \right \rvert \\
= & \sup_{t \in [\kappa_1,\kappa_2]}  \left \lvert \tilde B_H(t)+ C_1 \psi_{\tau}(t) \right \rvert.
\end{align*}
Note that this is the same limit as under a mean-shift with constant variance and thus, too, the asymptotic power is the same as in example \ref{AREMeanShift}.

The limiting distribution of the Cram\'{e}r-von Mises statistic is given by
\begin{align*}
 Z^2= & \ \sup_{t \in [\kappa_1,\kappa_2]}  \int \phi^3(x) \left ( \tilde B_H(t)+ (C_1+C_2 x) \psi_{\tau}(t) \right)^2 dx  \\
= & \ \sup_{t \in [\kappa_1,\kappa_2]} \left \{ \int \phi^3(x) dx  \left ( \tilde B_H(t)+ C_1 \psi_{\tau}(t) \right)^2 + C_2^2 \psi_{\tau}^2(t)  \int \phi^3(x)x^2 dx \right \},
\end{align*}
and for its asymptotic power we obtain
\begin{align}
\nonumber & \ P\left (Z^2> q^2_{1-\alpha,H} \int \phi^3(x) \ dx \right ) \\
= & \ P\left (Z^2>   q^2_{1-\alpha,H}  \int \phi^3(x) \ dx \ , \  \sup_{t \in [\kappa_1,\kappa_2]} \{ \tilde B_H(t) \} > q_{1-\alpha,H} \right ) \label{BewAREMeanVarGruen} \\
\nonumber & \ + P\left (Z^2> q^2_{1-\alpha,H} \int \phi^3(x) \ dx \ , \ \sup_{t \in [\kappa_1,\kappa_2]} \{ \tilde B_H(t) \} \leq q_{1-\alpha,H} \right  ).
\end{align}
First assume $\sup_t \{ \tilde B_H(t) \} \leq q= q_{1-\alpha,H}$ and consider $C_1^{\ast}$, given by
\begin{align*}
C_1^{\ast} = & \  f^{\ast} (C_1,C_2,q,\tau,\kappa_1,\kappa_2) \\
= & \ \min_{t \in [\kappa_1, \kappa_2]} \left \{ \frac {\sqrt{q^2+2qC_1\psi_{\tau}(t) + \left (C_1^2+C_2^2 (\int \psi^3(x)x^2 dx/ \int \phi^3(x) dx) \right) \psi_{\tau}^2(t)} - q} {\psi_{\tau}(t)} \right \}.
\end{align*}
Now $C_1^{\ast}$ is constructed in a way, such that \footnote{Here weed the restriction to $[\kappa_1,\kappa_2]$.}
\begin{align*}
C_1^{\ast} >C_1
\end{align*}
and for all $\omega \in \Omega$ with $\sup_t \tilde B_H(t;\omega) \leq q$
\begin{align*}
 Z^2 = & \ \sup_{t \in [\kappa_1,\kappa_2]} \left \{ \int \phi^3(x) dx  \left ( \tilde B_H(t)+ C_1 \psi_{\tau}(t) \right)^2 + C_2 \psi_{\tau}^2(t)  \int \phi^3(x)x^2 dx \right \} \\
> & \  \sup_{t \in [\kappa_1,\kappa_2]} \left \{ \int \phi^3(x) dx  \left ( \tilde B_H(t)+ C_1^{\ast} \psi_{\tau}(t) \right)^2  \right \}.
\end{align*}
If, on the other hand, $\sup_t \{ \tilde B_H(t) \} > q_{1-\alpha,H}$, then (because $C_1>0$) automatically $Z^2> q^2_{1-\alpha,H} \int \phi^3(x) dx$. Combining these two findings with (\ref{BewAREMeanVarGruen}) we can bound the asymptotic power from below by
\begin{align}
\nonumber & \ P\left (Z^2> q^2_{1-\alpha,H} \int \phi^3(x) \ dx \right ) \\
\nonumber = & \ P \left ( \sup_{t \in [\kappa_1,\kappa_2]}  \int \phi^3(x) \left ( \tilde B_H(t)+C_1^{\ast} \psi_{\tau}(t) \right)^2 \ dx > q^2_{1-\alpha,H} \int \phi^3(x) dx \right ) \\
\geq & \ P \left ( \sup_{t \in [\kappa_1,\kappa_2]} \left \lvert  \tilde B_H(t)+C_1^{\ast} \psi_{\tau}(t) \right \rvert  > q_{1-\alpha,H} \right ), \label{AsPowerBoundRot}
\end{align}
for $C_1^{\ast} > C_1$. 

Now we are ready to compute the ARE. To this end we chose different sample sizes for both test. In detail, $(n_k)_k$ for the Cram\'{e}r-von Mises test and $(m_k)_k$ for the CUSUM test. Moreover, the local alternatives are such that $G_{n_k}^{CvM}(x) = G_{m_k}^{CUSUM}(x) = G_k(x)$ for all $x\in \mathds R$, consequently
\begin{align*}
\mu^{(1)}_{n_k} = \mu^{(2)}_{m_k} = \mu_k \ \ \ \ \text{and} \ \ \ \ \sigma^{(1)}_{n_k} = \sigma^{(2)}_{m_k} = \sigma_k.
\end{align*}
For the CUSUM test, in order to achieve at least asymptotic power $\beta$, its limit distribution has to satisfy
\begin{align*}
P \left ( \sup_{t \in [0,1]} \lvert \tilde B_H(t) +C_1^{\ast} \psi_{\tau} (t) \rvert > q_{\alpha, H} \right ) \geq \beta.
\end{align*}
In other words, $C_1^{\ast} = \pi^{-1} (\beta)$, where
\begin{align*}
\pi (C^{\ast} ) = P \left ( \sup_{t \in [0,1]} \lvert \tilde B_H(t) +C_1^{\ast} \psi_{\tau} (t) \rvert > q_{\alpha, H} \right )
\end{align*}
and $\pi^{-1}$ is the generalized inverse. Therefore, the sample size of the CUSUM test has to be chosen such that
\begin{align}
C_1^{\ast} = \lim_{k \to \infty} \frac {m_k} {d_{m_k}} \mu^{(2)}_{m_k} = \lim_{k \to \infty} \frac {m_k} {d_{m_k}} \mu_k.
\end{align}
We also obtain (as $\mu_k$ and $(1-1/\sigma_k)$ are of the same order)
\begin{align*}
\lim_{k \to \infty} \frac {m_k} {d_{m_k}} \left ( 1- \frac 1 {\sigma^{(2)}_{m_k}} \right ) =\lim_{k \to \infty} \frac {m_k} {d_{m_k}} \left ( 1- \frac 1 {\sigma_k} \right ) = C_2^{\ast},
\end{align*}
for some $C_2^{\ast}>0$.

Next we will select the sample size for the Cram\'{e}r-von Mises test such that its asymptotic power might be bounded from below as in (\ref{AsPowerBoundRot}) (and therefore by $\beta$). To this end choose $C_1>0$, such that
\begin{align*}
\tilde f(C_1) = f(C_1,C_1 \frac {C_2^{\ast}} {C_1^{\ast}}, q,\tau, \kappa_1,\kappa_2) = C_1^{\ast}.
\end{align*}
The function $\tilde f \colon [0,\infty) \to [0,\infty)$ is monotone increasing, surjective and continuous (as the minimum is attained either in $\kappa_1$ or $\kappa_2$), therefore such an $C_1$ always can be found. By construction of the function $f$  it follows that $C_1<C_1^{\ast}$. Now let the sample size of the Cram\'{e}r-von Mises test satisfy
\begin{align}
C_1 = \lim_{k \to \infty} \frac {n_k} {d_{n_k}} \mu^{(1)}_{n_k} = \lim_{k \to \infty} \frac {n_k} {d_{n_k}} \mu_k.
\end{align}
Moreover, we observe
\begin{align*}
C_2 =& \ \lim_{k \to \infty} n_k/d_{n_k} (1-1/\sigma_k) \\
= & \ \lim_{k \to \infty} \left ( n_k/d_{n_k} \mu_k \right ) \left ( m_k / d_{m_k} \mu_k \right )^{-1} \left (m_k/ d_{m_k} (1-1/\sigma_k) \right )   \\
= & \ C_1 C_2^{\ast}/ C_1^{\ast}.
\end{align*}
Thus,
\begin{align*}
f(C_1,C_2, q,\tau, \kappa_1,\kappa_2)= C_1^{\ast},
\end{align*}
and we observe for the asymptotic power of the Cram\'{e}r-von Mises test
\begin{align*}
& P \left ( \sup_{t \in [\kappa_1,\kappa_2]} \left \{ \int \phi^3(x) \ dx  \left ( \tilde B_H(t)+ C_1 \psi_{\tau}(t) \right)^2 + C_2 \psi_{\tau}^2(t)  \int \phi^3(x)x^2 \ dx \right \} > q_{1-\alpha,H} \int \phi^3(x) dx \right ) \\
\geq & \ P \left ( \sup_{t \in [\kappa_1,\kappa_2]} \left \lvert  \tilde B_H(t)+C_1^{\ast} \psi_{\tau}(t) \right \rvert  > q_{1-\alpha,H} \right ) \\
\geq & \ \beta.
\end{align*}
So both test have (at least) asymptotic power $\beta$ against the local alternatives $(G,G_k,\tau)$. Finally,
\begin{align*}
\left ( \frac {m_k} {n_k} \right )^{1-H} = \frac {m_k} {d_{m_k}} \mu_k \frac {d_{n_k}} {n_k} \mu_k^{-1} \frac {L^{(1/2)}(m_k)} {L^{(1/2)}(n_k)}
\\ \rightarrow \frac {C_1^{\ast}} {C_1} >1,
\end{align*}
by construction of the sample sizes and the definition of slowly varying functions. Consequently 
\begin{align*}
ARE(CvM,CUSUM) = (C_1^{\ast}/C_1)^{1/(1-H)}>1.
\end{align*}
 In other words, the Cram\'{e}r-von Mises test is asymptotically more efficient, no matter how small the additional variance-change is.
\end{Bei}

\subsection{The empirical process of triangular arrays} \label{LimitTheoremsArrays}

Since the work of \citet{DeTa,DeTa2}, uniform reduction principles have become the main tool in the analysis of empirical processes of long-range dependent data. More precisely, the empirical process gets approximated only by the first term of its Hermite expansion (if the underlying process is not Gaussian other expansions are available). However, most results are investigated for stationary sequences. When considering $ G(X_1),  \dots , G(X_{ \lfloor n \tau \rfloor }), G_n(X_{ \lfloor n \tau \rfloor +1 }), \dots, G_n(X_n)$, the empirical process of the first $\lfloor n \tau \rfloor$ random variables can be approximated just as in \citet{DeTa}. In contrast the Hermite expansion of $1_{ \{ G_n(X_i) \leq x \}} - F_{(n)}(x)$ is
\begin{align*}
\sum_{q=m^{\ast}}^{\infty} \frac {J_{q,n}(x) } {q!} H_q(X_i).
\end{align*}
Two difficulties arise. First, $m^{\ast}$ might be smaller than $m$, the Hermite rank of $\{1_{ \{ G(\cdot) \leq x \}  }\}_{x \in \mathds R}$. Secondly, the coefficients $J_{q,n}(x)$ depend on $n$ and might converge uniformly to $0$. Thus, it is a priori not clear which term of the Hermite expansion is asymptotically dominant or if there are even more than one. The next result is a reduction principle that lays emphasis on this aspects. We will make use of it in the proof of Theorem \ref{ConvChangeAltA}, but is also of interest on its own.

\begin{The} \label{WeakRedPrinArray}
Let $\{G_n\}_n$ be a sequence of measurable functions and let $m(n)$ be the sequence of Hermite ranks of $\{1_{ \{ G_n(\cdot) \leq x\}} \}_{x \in \mathds R}$. Then, for any $m \in \mathds N$ with $m(n) \leq m < 1/D$ (for $n \geq n_0$),
\begin{align*}
P \left ( \sup_{t \in (0,1)} \sup_{ \ x \in \mathds R} \frac 1 {d_{n,m}}  \left \lvert \sum_{i=1}^{\lfloor nt \rfloor} (1_{\{ G_n(X_i) \leq x \}}  - \sum_{q=0}^m \frac {J_{q,n}(x)} {q!} H_q(X_i)  ) \right \rvert > \epsilon \right ) \leq Cn^{-\kappa}(1+ \epsilon^{-3}),
\end{align*}
where $C$ and $\kappa$ do not depend on $n$.
\end{The}

\begin{Bem} (i) Theorem \ref{WeakRedPrinArray} contains the reduction principle of \citet{DeTa} as a special case ($G_n(x) = G(x)$ and $m(n)=m$). \\
(ii) Note that $\{1_{ \{ G_n(\cdot) \leq x \}  } - F_{(n)}(x)\}_{x \in \mathds R}$ might has a Hermite rank smaller than $m$ (say $m^{\ast} < m$). Thus, one might expect $d_{n,m^{\ast}}^{-1}$ as normalization.  The weaker normalization $d_{n,m}^{-1}$ is however possible since the empirical process is approximated by additional terms of the Hermite expansion, in detail those up to $m$. \\
(iii) A similar result is given by \citet{Wu}, who considers linear long memory processes and even shows convergences with respect to a weighted supremum metric. Then again, he considers only the normal empirical process, while we also treat the sequential version. Moreover, we consider triangular arrays, which \citet{Wu} does not.
\end{Bem}

\begin{Kor} \label{KonvEmpProArray}
Let $\{G_n\}_n$ be sequence of measurable functions and let $m(n)$ be the sequence of Hermite ranks of $\{1_{ \{ G_n(\cdot) \leq x\}} \}_{x \in \mathds R}$. If further $m^{\ast} \leq m(n) \leq m < 1/D$ for all $n \geq n_0$ and
\begin{align*}
\frac {d_{n,q}} {d_{n,m}} \frac {J_{q,n}(x)} {q!} \rightarrow h_q(x) \ \ \forall q \in \{m^{\ast}, \dots, m\},
\end{align*}
uniformly in $x$, then
\begin{align*}
\frac 1 {d_{n,m}} \sum_{i=1}^{\lfloor nt \rfloor} (1_{\{ G_n(X_i) \leq x \}} - F_{(n)}(x)) \xrightarrow{\mathcal D} \sum_{q=m^{\ast}}^m h_q(x) Z_q(t).
\end{align*}
$(Z_{q,H}(t))_{t \in [0,1]}$ are uncorrelated, $q$th order Hermite processes. 
\end{Kor}

\begin{Bem} (i) Comparing the limit process of Corollary \ref{KonvEmpProArray} to that of Theorem \ref{ConvChangeAltA} it is apparent that multiple Hermite Processes are involved. This is not the case in Theorem \ref{ConvChangeAltA}. The reason is Assumption A2, which causes the Hermite coefficients $J_{m,n}(x)$ to converge rather fast. \\
(ii) The Hermite processes occurring in the limit are dependent, see Proposition 1 in \citet{BaTa}.
\end{Bem}

\begin{Bem} \label{BemerkungEigenschaftFunktionen} In view of the proof of Corollary \ref{KonvEmpProArray} it is important to note that the functions $h_q$ are uniform limits of the c\'{a}dl\'{a}g-functions $J_{m,n}(x)$  and hence elements of $D[- \infty, \infty]$. As a consequence they are also bounded (\citet{Pol}).
\end{Bem}

\begin{Bei} \label{EsgibtsolcheFunktion}
There are indeed sequences of functions $\{G_n \}_n$ that satisfy the conditions of Corollary \ref{KonvEmpProArray}. Consider again the functions from Example \ref{BeipsielChiQuadrat}, namely $G_n(x) = x^2(1_{x\geq 0 } +a_n 1_{x<0})$ with $a_n \rightarrow 1$ and $a_n\not =1$. Thus, we are in the situation of Theorem \ref{WeakRedPrinArray} with $m(n) =1$ for all $n \in \mathds N$. One obtains, $a_n \rightarrow 1$s
\begin{align*}
\sup_{x \in \mathds R} \lvert J_{2,n}(x) -J_2(x) \rvert \rightarrow 0,
\end{align*}
with 
\begin{align*}
J_2(x) = E[1_{ \{X_1^2  \leq x \}} (X_1^2-1)] = -2 \sqrt x \phi( \sqrt x) 1_{ \{ x\geq 0\} }.
\end{align*}
If in addition $a_n \sim n^{-D/2}L^{1/2}(n) \sim d_{n,2}/d_{n,1}$, then
\begin{align*}
 \sup_{x} \left \lvert \frac {d_{n,1}} {d_{n,2}}J_{1,n}(x) - C x \phi( \sqrt x) 1_{x\geq0} \right \rvert \rightarrow 0,
\end{align*}
for some constant $C$ depending on $D$ only. Corollary \ref{KonvEmpProArray} then holds with $m=2$, $m^{\ast}=1$, $h_1(x) = C x \phi( \sqrt x) 1_{x\geq0}$ and $h_2(x) =J_2(x)/2$. 
\end{Bei}

\section{Simulation Study} \label{SimStudy}

\subsection{Fractional Gaussian Noise}

Consider a mean-shift in Gaussian data. Then Example \ref{MeanShift} states that the Cram\'{e}r-von Mises test (and the Kolmogorov-Smirnov test) are asymptotically as efficient as the CUSUM test. The goal of this simulation study is to examine whether this theoretical and asymptotic result carry over to the finite sample performance of the tests. We will consider samples of size $50$ to $400$. For these situations the approximation of the empirical process by its semi-degenerate limit process is quite inaccurate. The empirical size of the Cram\'{e}r-von Mises test will be therefore much larger than the nominal size, if critical values are deduced from the asymptotic distribution. Instead we simulate $J=1000$ Gaussian time series 
\begin{align*}
X_{j,1}, \dots X_{j,n} \ \ \ \ j=1, \dots J
\end{align*}
with Hurst coefficient $H$. In the simulation study we will use \textsl{fractional Gaussian noise} for this sequences. Subsequently, a Cram\'{e}r-von Mises statistic is calculated for each of the $J=1000$ Gaussian series, in detail
\begin{align*}
S_{n,j} = \max_{1 \leq k <n} \int_{x \in \mathds R} \left ( \sum_{i=1}^k 1_{ \{ X_{j,i} \leq x \} } - \frac{k}{n} \sum_{i=1}^n 1_{ \{ X_{j,i} \leq x \} } \right)^2 \ d \hat F_{n,j}(x) \ \ \ \ j=1, \dots J.
\end{align*}

We then use the empirical quantiles of $\{ S_{n,j} \}_{j=1}^J$ as critical values. The Cram\'{e}r-von Mises statistic is invariant under monotone transformations of the data (as is the Kolmogorov-Smirnov statistic). Hence the critical values are valid if our observations are monotone transformations of Gaussian data. We note that this is a strong assumption and that an accurate approximation of the empirical process for general long-range dependent data is an issue of future research. The CUSUM statistic is not invariant under monotone transformations. Therefore, the Wilcoxon change-point test is considered additionally. 

In the first part of the simulation study we treat realizations of a Gaussian process $X_1,\dots , X_n$ given by \textsl{fractional Gaussian noise}). For the implementation we have used the function \texttt{fgnSim} from the \texttt{R}-package \texttt{fArma}. Eventually a change is added by $Y_i = X_i + \mu 1_{\{i> \lfloor n \tau \rfloor \}}$ and the three mentioned change-point tests are applied to $Y_1,\dots ,Y_n$.

\begin{table*}
\centering
\caption{Empirical power, $H$ assumed to be known, size of level shift $\mu=1$,  relative change positions $\tau =0.2$ and $\tau =0.5$.}
\label{TabelleHassumedtobeknown}
\begin{tabular}{l c c c c c c c c}
  \toprule
& \multicolumn{8}{c}{ Relative change position $\tau =0.2$}\\ 
\midrule
 & \multicolumn{4}{c}{$H=0.6$} & \multicolumn{4}{c}{$H=0.7$}\\ 
 \cmidrule(lr){2-5} \cmidrule(lr){6-9}
 $n$ & 50  & 100 & 250 & 400 & 50  & 100 & 250 & 400 \\
  \midrule
S & 0.196 & 0.566 & 0.854 & 0.970 & 0.158 & 0.215 & 0.547 & 0.689 \\ 
  W & 0.263 & 0.525 & 0.910 & 0.983 & 0.201 & 0.233 & 0.501 & 0.636 \\ 
  C & 0.288 & 0.666 & 0.933 & 0.986 & 0.276 & 0.284 & 0.555 & 0.769 \\ 
    \midrule
  & \multicolumn{4}{c}{$H=0.8$} & \multicolumn{4}{c}{$H=0.9$}\\ 
 \cmidrule(lr){2-5} \cmidrule(lr){6-9}
 $n$ & 50  & 100 & 250 & 400 & 50  & 100 & 250 & 400 \\
  \midrule
  S & 0.101 & 0.221 & 0.241 & 0.350 & 0.057 & 0.098 & 0.236 & 0.167 \\ 
  W& 0.089 & 0.156 & 0.264 & 0.383 & 0.089 & 0.116 & 0.191 & 0.147 \\ 
  C & 0.171 & 0.234 & 0.348 & 0.349 & 0.164 & 0.127 & 0.239 & 0.223 \\ 
   \midrule
\\
& \multicolumn{8}{c}{ Relative change position $\tau =0.5$}\\ 
\midrule
 & \multicolumn{4}{c}{$H=0.6$} & \multicolumn{4}{c}{$H=0.7$}\\ 
 \cmidrule(lr){2-5} \cmidrule(lr){6-9}
 $n$ & 50  & 100 & 250 & 400 & 50  & 100 & 250 & 400 \\
  \midrule
S & 0.664 & 0.919 & 1.000 & 1.000 & 0.524 & 0.682 & 0.925 & 0.970 \\ 
  W & 0.621 & 0.930 & 0.998 & 1.000 & 0.513 & 0.742 & 0.906 & 0.967 \\ 
  C & 0.733 & 0.918 & 0.997 & 0.999 & 0.599 & 0.717 & 0.919 & 0.960 \\ 
      \midrule
  & \multicolumn{4}{c}{$H=0.8$} & \multicolumn{4}{c}{$H=0.9$}\\ 
 \cmidrule(lr){2-5} \cmidrule(lr){6-9}
 $n$ & 50  & 100 & 250 & 400 & 50  & 100 & 250 & 400 \\
  \midrule
  S & 0.418 & 0.504 & 0.655 & 0.830 & 0.359 & 0.461 & 0.475 & 0.526 \\ 
  W & 0.374 & 0.485 & 0.674 & 0.770 & 0.387 & 0.430 & 0.578 & 0.587 \\ 
  C & 0.400 & 0.553 & 0.673 & 0.766 & 0.393 & 0.499 & 0.522 & 0.553 \\ 
   \bottomrule
\end{tabular}
\end{table*}

\begin{table*}
\centering
\caption{Empirical size, estimated Hurst coefficient.}
\label{TabelleEmpSizeEstimatedCoefficient}
\begin{tabular}{l c c c c c c c c}
  \toprule
 & \multicolumn{4}{c}{$H=0.6$} & \multicolumn{4}{c}{$H=0.7$}\\ 
 \cmidrule(lr){2-5} \cmidrule(lr){6-9}
 $n$ & 50  & 100 & 250 & 400 & 50  & 100 & 250 & 400 \\
  \midrule
$S_{\hat H}$ & 0.067 & 0.088 & 0.065 & 0.058 & 0.080 & 0.076 & 0.063 & 0.043 \\ 
   $S_{\hat H_{\hat k}}$ & 0.082 & 0.105 & 0.085 & 0.083 & 0.122 & 0.143 & 0.107 & 0.081 \\ 
  $W_{\hat H}$ & 0.067 & 0.058 & 0.054 & 0.057 & 0.081 & 0.074 & 0.061 & 0.035 \\ 
   $W_{\hat H_{\hat k}}$ & 0.070 & 0.100 & 0.102 & 0.078 & 0.122 & 0.118 & 0.100 & 0.081\\ 
  $C_{\hat H}$ & 0.071 & 0.072 & 0.064 & 0.046 & 0.111 & 0.080 & 0.056 & 0.056 \\ 
   $C_{\hat H_{\hat k}}$ & 0.085 & 0.090 & 0.103 & 0.081 & 0.116 & 0.127 & 0.095 & 0.078 \\ 
  \midrule
  & \multicolumn{4}{c}{$H=0.8$} & \multicolumn{4}{c}{$H=0.9$}\\ 
 \cmidrule(lr){2-5} \cmidrule(lr){6-9}
 $n$ & 50  & 100 & 250 & 400 & 50  & 100 & 250 & 400 \\
  \midrule
  $S_{\hat H}$ & 0.094 & 0.104 & 0.062 & 0.063 & 0.075 & 0.070 & 0.092 & 0.080 \\ 
   $S_{\hat H_{\hat k}}$ & 0.136 & 0.141 & 0.099 & 0.077 & 0.087 & 0.127 & 0.118 & 0.112  \\ 
  $W_{\hat H}$ & 0.085 & 0.074 & 0.056 & 0.059 & 0.073 & 0.071 & 0.097 & 0.079 \\ 
   $W_{\hat H_{\hat k}}$ & 0.129 & 0.137 & 0.087 & 0.074 & 0.098 & 0.146 & 0.115 & 0.103 \\ 
  $C_{\hat H}$ & 0.165 & 0.101 & 0.051 & 0.046 & 0.309 & 0.257 & 0.112 & 0.075 \\ 
   $C_{\hat H_{\hat k}}$& 0.218 & 0.127 & 0.094 & 0.072 & 0.137 & 0.130 & 0.083 & 0.063 \\ 
 \bottomrule
\end{tabular}
\end{table*}

If the Hurst-coefficient is assumed to be known, the empirical size of the tests naturally equals the nominal one, due to the construction of the critical values. The empirical power of Cram\'{e}r-von Mises (denoted by $S_n$), Wilcoxon (denoted by $W_n$) and CUSUM test (denoted by $C_n$) is displayed in Table \ref{TabelleHassumedtobeknown}. If the change occurs in the middle of the observation period, the three tests are showing almost exactly the same performance, which matches the theoretical results. For early changes (after $20 \%$ of the observations) the CUSUM test is slightly more accurate than the other tests. Depending on sample size and strength of dependence, either the Cram\'{e}r-von Mises or the Wilcoxon test might be second best. \\

\begin{table*}
\centering
\caption{Empirical Power, estimated Hurst coefficient, size of level shift $\mu=1$,  relative change positions $\tau =0.2$ and $\tau =0.5$.}
\label{TablleEmpPower021}
\begin{tabular}{l c c c c c c c c}
  \toprule
& \multicolumn{8}{c}{ Relative change position $\tau =0.2$}\\ 
\midrule
 & \multicolumn{4}{c}{$H=0.6$} & \multicolumn{4}{c}{$H=0.7$}\\ 
 \cmidrule(lr){2-5} \cmidrule(lr){6-9}
 $n$ & 50  & 100 & 250 & 400 & 50  & 100 & 250 & 400 \\
  \midrule
$S_{\hat H}$ & 0.178 & 0.263 & 0.597 & 0.828 & 0.189 & 0.177 & 0.285 & 0.404 \\ 
  $S_{\hat H_{\hat k}}$ & 0.248 & 0.491 & 0.834 & 0.946 & 0.260 & 0.377 & 0.528 & 0.620 \\ 
  $W_{\hat H}$ & 0.190 & 0.298 & 0.625 & 0.811 & 0.146 & 0.193 & 0.278 & 0.386 \\ 
  $W_{\hat H_{\hat k}}$ & 0.291 & 0.532 & 0.850 & 0.943 & 0.237 & 0.364 & 0.488 & 0.649 \\ 
  $C_{\hat H}$ & 0.289 & 0.413 & 0.706 & 0.896 & 0.312 & 0.266 & 0.379 & 0.547 \\ 
  $C_{\hat H_{\hat k}}$ & 0.378 & 0.596 & 0.874 & 0.972 & 0.312 & 0.422 & 0.585 & 0.721 \\ 
  \midrule
  & \multicolumn{4}{c}{$H=0.8$} & \multicolumn{4}{c}{$H=0.9$}\\ 
 \cmidrule(lr){2-5} \cmidrule(lr){6-9}
 $n$ & 50  & 100 & 250 & 400 & 50  & 100 & 250 & 400 \\
  \midrule
  $S_{\hat H}$ & 0.115 & 0.125 & 0.156 & 0.185 & 0.096 & 0.105 & 0.142 & 0.163 \\ 
  $S_{\hat H_{\hat k}}$ & 0.208 & 0.265 & 0.287 & 0.296 & 0.160 & 0.213 & 0.208 & 0.234 \\ 
  $W_{\hat H}$ & 0.150 & 0.146 & 0.161 & 0.181 & 0.094 & 0.126 & 0.143 & 0.158 \\ 
  $W_{\hat H_{\hat k}}$ & 0.214 & 0.271 & 0.281 & 0.322 & 0.137 & 0.188 & 0.237 & 0.234 \\ 
  $C_{\hat H}$ & 0.405 & 0.321 & 0.270 & 0.305 & 0.539 & 0.442 & 0.367 & 0.328 \\ 
  $C_{\hat H_{\hat k}}$ & 0.313 & 0.321 & 0.383 & 0.429 & 0.418 & 0.350 & 0.311 & 0.313 \\ 
\midrule
\\
& \multicolumn{8}{c}{ Relative change position $\tau =0.5$}\\ 
\midrule
& \multicolumn{4}{c}{$H=0.6$} & \multicolumn{4}{c}{$H=0.7$}\\ 
 \cmidrule(lr){2-5} \cmidrule(lr){6-9}
 $n$ & 50  & 100 & 250 & 400 & 50  & 100 & 250 & 400 \\
  \midrule
$S_{\hat H}$ & 0.541 & 0.759 & 0.985 & 0.999 & 0.412 & 0.557 & 0.814 & 0.904 \\ 
   $S_{\hat H_{\hat k}}$ & 0.614 & 0.860 & 0.990 & 0.999 & 0.539 & 0.710 & 0.881 & 0.952 \\ 
  $W_{\hat H}$ & 0.594 & 0.811 & 0.984 & 0.999 & 0.441 & 0.564 & 0.809 & 0.902 \\ 
   $W_{\hat H_{\hat k}}$ & 0.609 & 0.877 & 0.991 & 1.000 & 0.550 & 0.717 & 0.878 & 0.950 \\ 
  $C_{\hat H}$ & 0.677 & 0.819 & 0.988 & 1.000 & 0.584 & 0.671 & 0.850 & 0.925 \\ 
   $C_{\hat H_{\hat k}}$ & 0.694 & 0.905 & 0.995 & 0.998 & 0.567 & 0.760 & 0.920 & 0.953 \\ 
  \midrule
  & \multicolumn{4}{c}{$H=0.8$} & \multicolumn{4}{c}{$H=0.9$}\\ 
 \cmidrule(lr){2-5} \cmidrule(lr){6-9}
  $S_{\hat H}$ & 0.373 & 0.403 & 0.535 & 0.640 & 0.337 & 0.428 & 0.472 & 0.549 \\ 
   $S_{\hat H_{\hat k}}$ & 0.454 & 0.563 & 0.649 & 0.711 & 0.412 & 0.480 & 0.582 & 0.604 \\ 
  $W_{\hat H}$ & 0.369 & 0.413 & 0.563 & 0.648 & 0.357 & 0.387 & 0.506 & 0.536 \\ 
   $W_{\hat H_{\hat k}}$ & 0.443 & 0.566 & 0.662 & 0.710 & 0.433 & 0.528 & 0.544 & 0.599 \\ 
  $C_{\hat H}$ & 0.576 & 0.584 & 0.655 & 0.706 & 0.659 & 0.622 & 0.631 & 0.637 \\ 
   $C_{\hat H_{\hat k}}$& 0.535 & 0.599 & 0.716 & 0.760 & 0.604 & 0.574 & 0.637 & 0.638 \\ 
   \bottomrule
\end{tabular}
\end{table*}

\subsection{Unknown Hurst coefficient}

In applications the true Hurst coefficient $H$ is unknown, and in the following we will consider two different estimators. The first is the local Whittle estimator (denoted by $\hat H$) with bandwidth parameter $m=\lfloor n^{2/3} \rfloor$, see \citet{Kuen}. However, if there is actually a change in the data, the local Whittle estimator is known to be biased. For the second estimator we therefore divide the observations into two subsamples
\begin{align*}
X_1, \dots ,X_{\hat k} \ \ \ \ \text{and} \ \ \ \ X_{\hat k +1}, \dots , X_n
\end{align*}
and estimate $H$ on each set, using again the local Whittle estimator. Finally the new estimator is given by $\hat H_{ \hat k} =  \hat k /n \hat H_1 + (n- \hat k)/n \hat H_2$. Here $\hat k$ is the natural change-point estimator, associated to each test. For example, in case of the Cram\'{e}r-von Mises test we use
\begin{align*}
\hat k = \min \left \{ 1 \leq k \leq n-1 \ \vert \ U_{k,n} = \max_{1\leq k \leq n-1} U_{k,n} \right \},
\end{align*}
where 
\vspace{-10pt}
\begin{align*}
U_{k,n} =  \int_{x \in \mathds R} \left ( \sum_{i=1}^k 1_{ \{ X_{i} \leq x \} } - \frac{k}{n} \sum_{i=1}^n 1_{ \{ X_{i} \leq x \} } \right)^2 \ d \hat F_n(x).
\end{align*}

\begin{table*}
\centering
\caption{Empirical Power, estimated Hurst coefficient, relative change position $\tau=0.5$, level shift of size $\mu=1$ and change in variance from $\sigma^2=1$ to $\sigma^2_0=5/4$, nominal size $\alpha=0.05$.}
\label{TablleEmpPowerVariance}
\begin{tabular}{l c c c c c c c c}
  \toprule
 & \multicolumn{4}{c}{$H=0.6$} & \multicolumn{4}{c}{$H=0.7$}\\ 
 \cmidrule(lr){2-5} \cmidrule(lr){6-9}
 $n$ & 50  & 100 & 250 & 400 & 50  & 100 & 250 & 400 \\
  \midrule
$S_{\hat H}$ & 0.701 & 0.931 & 1.000 & 1.000 & 0.606 & 0.772 & 0.963 & 0.996 \\ 
   $S_{\hat H_{\hat k}}$ & 0.878 & 0.986 & 1.000 & 1.000 & 0.817 & 0.973 & 0.999 & 1.000 \\ 
  $W_{\hat H}$ & 0.609 & 0.812 & 0.989 & 1.000 & 0.572 & 0.593 & 0.879 & 0.952 \\ 
   $W_{\hat H_{\hat k}}$ & 0.734 & 0.973 & 1.000 & 1.000 & 0.705 & 0.907 & 0.988 & 0.998 \\ 
  $C_{\hat H}$ & 0.660 & 0.899 & 0.999 & 1.000 & 0.529 & 0.724 & 0.938 & 0.983 \\ 
   $C_{\hat H_{\hat k}}$ & 0.588 & 0.916 & 1.000 & 1.000 & 0.507 & 0.806 & 0.983 & 0.998 \\ 
  \midrule
  & \multicolumn{4}{c}{$H=0.8$} & \multicolumn{4}{c}{$H=0.9$}\\ 
 \cmidrule(lr){2-5} \cmidrule(lr){6-9}
  $S_{\hat H}$ & 0.466 & 0.636 & 0.824 & 0.898 & 0.507 & 0.599 & 0.755 & 0.797 \\ 
   $S_{\hat H_{\hat k}}$ & 0.762 & 0.942 & 0.983 & 1.000 & 0.824 & 0.960 & 0.993 & 0.993 \\ 
  $W_{\hat H}$ & 0.597 & 0.568 & 0.669 & 0.727 & 0.718 & 0.645 & 0.582 & 0.634 \\ 
   $W_{\hat H_{\hat k}}$ & 0.616 & 0.850 & 0.944 & 0.981 & 0.562 & 0.823 & 0.926 & 0.953 \\ 
  $C_{\hat H}$ &  0.445 & 0.601 & 0.806 & 0.853 & 0.438 & 0.551 & 0.731 & 0.762 \\ 
   $C_{\hat H_{\hat k}}$& 0.443 & 0.635 & 0.852 & 0.937 & 0.460 & 0.576 & 0.781 & 0.847 \\ 
   \bottomrule
\end{tabular}
\end{table*}

\begin{table*}
\centering
\caption{Empirical Power, estimated Hurst coefficient, relative change position $\tau=0.5$,$G_1(x)=x^2$, $G_2(x) = x^2+x/2+1/2$, nominal size $\alpha=0.05$, $H$ is the Hurst coefficient of the underlying Gaussian.}
\label{TablleEmpPowerChiSquare}
\begin{tabular}{l c c c c c c c c}
  \toprule
 & \multicolumn{4}{c}{$H=0.6$} & \multicolumn{4}{c}{$H=0.7$}\\ 
 \cmidrule(lr){2-5} \cmidrule(lr){6-9}
 $n$ & 50  & 100 & 250 & 400 & 50  & 100 & 250 & 400 \\
  \midrule
 $S_{\hat H}$ & 0.535 & 0.827 & 0.983 & 0.999 & 0.487 & 0.758 & 0.957 & 0.988 \\ 
  $S_{\hat H_{\hat k}}$ & 0.494 & 0.815 & 0.992 & 0.999 & 0.479 & 0.750 & 0.968 & 0.995 \\ 
  $W_{\hat H}$ & 0.480 & 0.743 & 0.990 & 0.998 & 0.430 & 0.632 & 0.924 & 0.986 \\ 
  $W_{\hat H_{\hat k}}$ & 0.420 & 0.735 & 0.986 & 1.000 & 0.407 & 0.657 & 0.933 & 0.987 \\ 
  $C_{\hat H}$ & 0.424 & 0.616 & 0.853 & 0.958 & 0.399 & 0.547 & 0.745 & 0.858 \\ 
  $C_{\hat H_{\hat k}}$ & 0.387 & 0.569 & 0.828 & 0.920 & 0.390 & 0.546 & 0.755 & 0.891 \\ 
  \midrule
  & \multicolumn{4}{c}{$H=0.8$} & \multicolumn{4}{c}{$H=0.9$}\\ 
 \cmidrule(lr){2-5} \cmidrule(lr){6-9}
  $S_{\hat H}$  & 0.461 & 0.670 & 0.825 & 0.885 & 0.424 & 0.522 & 0.607 & 0.614 \\ 
   $S_{\hat H_{\hat k}}$ & 0.474 & 0.680 & 0.860 & 0.934 & 0.507 & 0.589 & 0.698 & 0.704 \\ 
  $W_{\hat H}$  & 0.376 & 0.537 & 0.670 & 0.773 & 0.350 & 0.369 & 0.443 & 0.438 \\ 
   $W_{\hat H_{\hat k}}$ & 0.418 & 0.555 & 0.738 & 0.828 & 0.458 & 0.496 & 0.488 & 0.511 \\ 
  $C_{\hat H}$  & 0.374 & 0.464 & 0.564 & 0.596 & 0.384 & 0.352 & 0.343 & 0.352 \\ 
    $C_{\hat H_{\hat k}}$ & 0.440 & 0.538 & 0.655 & 0.772 & 0.397 & 0.491 & 0.562 & 0.579 \\ 
   \bottomrule
\end{tabular}
\end{table*}

Consistency of this estimator was shown in \citet{NiHaWyZh}. \citet{HoKo} verified consistency for the analogous CUSUM-based estimator.

Empirical size and empirical power of the tests under unknown $H$ are displayed in tables \ref{TabelleEmpSizeEstimatedCoefficient} and \ref{TablleEmpPower021}. Let us first compare the impact of the different estimators $\hat H$ and $\hat H_{\hat k}$ on the finite sample performance of the Cram\'{e}r-von Mises test. If we use the classical local Whittle estimator, the empirical size of the test is quite accurate and even matches the nominal size for $n=400$ and $H \leq 0.8$. However, there is a loss in the empirical power. The power performance is much better, if the local Whittle estimator is modified. Actually there is no loss in power if compared to the case where $H$ was assumed to be known. Then again, the probability of a false rejection is higher than $\alpha =0.05$, so the test is quite liberal.

Next we compare Cram\'{e}r-von Mises, Wilcoxon and CUSUM test. The empirical size of the three tests is similar, no matter which estimator we choose and which situation we assume (sample size, Hurst coefficient), see Table \ref{TabelleEmpSizeEstimatedCoefficient}. 

In terms of empirical power the Cram\'{e}r-von Mises and Wilcoxon test give similar results with the CUSUM test being slightly ahead for $\tau =1/2$ and being clearly advantageous for early changes $\tau =1/5$ (see Table \ref{TablleEmpPower021}). 

We have to keep in mind that CUSUM and Wilcoxon test are designed to detect changes in the mean. On the contrary, the Cram\'{e}r-von Mises test is a so called omnibus test and has power against arbitrary changes in the marginal distribution. 

Therefore, we consider another situation, with the mean-shift being now accompanied by a small change in the variance. In detail, 
\begin{align*}
Y_i = \begin{cases}
X_i & \text{for  } i \leq \lfloor n \tau \rfloor, \\
\sigma X_i + \mu & \text{for  } i > \lfloor n \tau \rfloor,
\end{cases}
\end{align*}
for Gaussian $\{X_i\}_{i \geq 1}$. The theoretic result from Example \ref{AREMeanVarChange} indicates that in this scenario the Cram\'{e}r-von Mises test should be advantageous. In fact, for all combinations of sample size $n$ and Hurst coefficient $H$ the empirical power against this change is always higher than the power against a mean-shift under constant variance. Moreover, the Cram\'{e}r-von Mises test has clearly higher power then CUSUM and Wilcoxon test, which matches the theoretical findings of Example \ref{AREMeanVarChange}.

Moreover, we consider the change-point problem (based on non-monotone transformations)
\begin{align*}
Y_i = \begin{cases}
X_i^2 & \text{for  } i \leq \lfloor n \tau \rfloor, \\
 X_i^2 +aX_i + \mu & \text{for  } i > \lfloor n \tau \rfloor,
\end{cases}
\end{align*}
corresponding to a situation in which mean, variance, skewness and the Hermite rank change (see Example \ref{BeispielKomplizierterLimes}). Table \ref{TablleEmpPowerChiSquare} displays the empirical power of the three tests against this alternative and the picture is quite clear. The Cram\'{e}r-von Mises test has the highest power for all combinations of $H$ and $n$, while the Wilcoxon test is second best. Also note that the Hermite rank of the pre-change random variables is $m=2$. Consequently, these observations are short-range dependent for $H < 0.75$. 

\subsection{$farima(0,d,0)$-processes}

\begin{table*}
\centering
\caption{Empirical size and power for $farima(0,0.2,0)$-sequences, Hurst coefficient is estimated, nominal size $\alpha=0.05$.}
\label{TabelleEmpSizeEstimatedCoefficientFarima}
\begin{tabular}{l c c c c c c c c}
  \toprule
 & \multicolumn{4}{c}{No change} & \multicolumn{4}{c}{Mean-shift $\mu=1$}\\ 
 \cmidrule(lr){2-5} \cmidrule(lr){6-9}
 $n$ & 50  & 100 & 250 & 400 & 50  & 100 & 250 & 400 \\
  \midrule
$S_{\hat H}$ & 0.036 & 0.058 & 0.069 & 0.056 & 0.276 & 0.490 & 0.832 & 0.920 \\ 
   $S_{\hat H_{\hat k}}$ & 0.167 & 0.148 & 0.119 & 0.129 & 0.520 & 0.737 & 0.939 & 0.986 \\ 
  $W_{\hat H}$ & 0.067 & 0.139 & 0.086 & 0.074 & 0.601 & 0.882 & 0.968 & 0.730 \\ 
   $W_{\hat H_{\hat k}}$ & 0.247 & 0.189 & 0.160 & 0.153 & 0.573 & 0.711 & 0.934 & 0.980\\ 
  $C_{\hat H}$ & 0.104 & 0.061 & 0.053 & 0.048 &0.281 & 0.479 & 0.836 & 0.945  \\ 
   $C_{\hat H_{\hat k}}$ & 0.212 & 0.202 & 0.158 & 0.114 & 0.499 & 0.682 & 0.937 & 0.977 \\ 
\bottomrule
\end{tabular}
\end{table*}

For Gaussian long memory processes beyond fractional Gaussian noise, not only the Hurst coefficient determines the normalization. Instead it is given by \begin{align*}
d_n=n^HL^{1/2}(n) (H(2H-1))^{1/2}),
\end{align*}
see (\ref{DefinitionNormalizationBlau}). In this study we assume, as $n \to \infty$,  $L(n) \rightarrow C$, which is quite common in the literature. For fractional Gaussian noise, $C=H(2H-1)$ so the two factors just cancel out. In general the constant $C$ is given through the limit
\begin{align*}
\rho(k) k^{2-2H} \rightarrow C,
\end{align*}
as $k \to \infty$. We suggest an estimator for $C$ (which is quite heuristic) by:
\begin{align}
\hat C = \frac 1 K \sum_{k=1}^K \hat \rho(k) k^{2-2 \hat H}, \label{SchaetzerKonstante}
\end{align}
with $\hat H$ being one of the two estimators from above. Finally, we use the normalization
\begin{align*}
\hat d_n = n^{\hat H}\hat C^{1/2}(\hat H(2 \hat H-1))^{1/2}.
\end{align*}
The estimator $\hat C$ in (\ref{SchaetzerKonstante}) is only defined under long memory, that is $H>0.5$ (or in this situation $\hat H>0.5$). Therefore, we modify both estimators by considering $\max(\hat H, 0.501)$ instead of $\hat H$. The effect of this modification on short memory processes will be seen in the next section.

However, for $farima(0,d,0)$-sequences it seems to work quite well, see Table \ref{TabelleEmpSizeEstimatedCoefficientFarima}. Note that critical values are still deduced from fractional Gaussian noise. The finite sample performance (under the hypothesis as well as under a mean-shift) is very similar to the case where the data comes from fractional Gaussian noise. Meaning, the Cram\'{e}r-von Mises test has good properties and the different tests yield very similar results, again matching the theoretic findings. 

\subsection{Short-range dependent effects}

\begin{table*}
\centering
\caption{Empirical size and power for $farima(1,0.2,0)$-sequences with AR-coefficient $a_1=0.4$, Hurst coefficient is estimated, nominal size $\alpha=0.05$.}
\label{TabelleEmpSizeEstimatedCoefficientFarima1}
\begin{tabular}{l c c c c c c c c}
  \toprule
 & \multicolumn{4}{c}{No change} & \multicolumn{4}{c}{Mean-shift $\mu=1$}\\ 
 \cmidrule(lr){2-5} \cmidrule(lr){6-9}
 $n$ & 50  & 100 & 250 & 400 & 50  & 100 & 250 & 400 \\
  \midrule
$S_{\hat H}$ & 0.032 & 0.021 & 0.030 & 0.034 & 0.154 & 0.183 & 0.329 & 0.433 \\ 
   $S_{\hat H_{\hat k}}$ & 0.126 & 0.065 & 0.062 & 0.052 & 0.316 & 0.325 & 0.415 & 0.489 \\ 
  $W_{\hat H}$ & 0.038 & 0.009 & 0.003 & 0.007 & 0.157 & 0.158 & 0.192 & 0.266 \\ 
   $W_{\hat H_{\hat k}}$ & 0.183 & 0.048 & 0.015 & 0.017 & 0.363 & 0.313 & 0.301 & 0.387 \\ 
  $C_{\hat H}$ & 0.445 & 0.288 & 0.124 & 0.081 & 0.625 & 0.615 & 0.605 & 0.689 \\ 
   $C_{\hat H_{\hat k}}$ & 0.422 & 0.302 & 0.138 & 0.093 & 0.592 & 0.613 & 0.656 & 0.710 \\ 
\bottomrule
\end{tabular}
\end{table*}

\begin{table*}
\centering
\caption{Empirical size and power for $AR(1)$-sequences with AR-coefficient $a_1=0.6$, Hurst coefficient is estimated, nominal size $\alpha=0.05$.}
\label{TabelleEmpSizeEstimatedCoefficientAR1}
\begin{tabular}{l c c c c c c c c}
  \toprule
 & \multicolumn{4}{c}{No change} & \multicolumn{4}{c}{Mean-shift $\mu=1$}\\ 
 \cmidrule(lr){2-5} \cmidrule(lr){6-9}
 $n$ & 50  & 100 & 250 & 400 & 50  & 100 & 250 & 400 \\
  \midrule
$S_{\hat H}$ & 0.016 & 0.008 & 0.004 & 0.002 & 0.135 & 0.150 & 0.324 & 0.576 \\ 
  $S_{\hat H_{\hat k}}$ & 0.097 & 0.030 & 0.012 & 0.004 & 0.282 & 0.260 & 0.374 & 0.556 \\ 
  $W_{\hat H}$ & 0.036 & 0.004 & 0.000 & 0.000 & 0.149 & 0.106 & 0.106 & 0.213 \\ 
  $W_{\hat H_{\hat k}}$ & 0.141 & 0.018 & 0.000 & 0.000 & 0.336 & 0.249 & 0.235 & 0.343 \\ 
  $C_{\hat H}$ & 0.467 & 0.279 & 0.017 & 0.006 & 0.641 & 0.592 & 0.606 & 0.721 \\ 
  $C_{\hat H_{\hat k}}$ & 0.412 & 0.219 & 0.026 & 0.006 & 0.654 & 0.631 & 0.678 & 0.794 \\ 
\bottomrule
\end{tabular}
\end{table*}

Finally, we have considered deviations from purely LRD sequences by simulating $farima(1,d,0)$-time series and short memory AR($1$)-processes. 

First, we have applied the tests to $farima(0,d,1)$-sequences, which are still long-range dependent. Table \ref{TabelleEmpSizeEstimatedCoefficientFarima1} indicates that the empirical power of the Cram\'{e}r-von Mises test is less than in the case of $farima(0,d,0)$-processes. However, the test works principally well, meaning that the power increases with the number of observations while the empirical size stays close to the nominal size. For CUSUM and Wilcoxon test this seems to be not the case.
\\\\
For the (purely short-range dependent) AR($1$)-processes we make two observations: First, due to the assumption of LRD ($H> 0.5$) the normalization is too strong and the statistics converge to $0$, at least under stationarity. If the structural change is big enough, the tests might still detect the change (see Table \ref{TabelleEmpSizeEstimatedCoefficientAR1}). However, there is a certain loss in power. 

Secondly, Cram\'{e}r-von Mises test and CUSUM test are showing a quite different finite sample performance. While under LRD (in concordance with the theory) their empirical size and power is always very similar, we now observe situations where the Cram\'{e}r-von Mises test has empirical power $0.374$ and the CUSUM test $0.678$, see the results in Table \ref{TabelleEmpSizeEstimatedCoefficientAR1}. Again, this matches the theoretical fact that under short memory the tests show a different asymptotic behavior.

\section{Proofs of the main results} \label{SectionProofs}

\subsection{Proof of Theorem \ref{WeakRedPrinArray} and Corollary \ref{KonvEmpProArray}}

It is the goal to approximate the sequential empirical process by a linear combination of multiple partial sum processes. The indicator function $1_{\{G_n(X_j) \leq x \}}$ has the Hermite expansion 
\begin{align*}
1_{\{G_n(X_j) \leq x \}} = \sum_{q=0}^{\infty} \frac {J_{q,n}(x)} {q!} H_q(X_j).
\end{align*}
Remind that $J_{q,n}(x) = E [1_{\{G_n(X_j) \leq x \}} H_q(X_j)]$ and especially $J_{0,n}(x) = P(G_n(X_j) \leq x) = F_{(n)}(x)$. Now let $L_{m,n,j}(x)$ be the Hermite expansion up to $m$, in detail
\begin{align}
L_{m,n,j}(x) = \sum_{q=0}^{m} \frac {J_{q,n}(x)} {q!} H_q(X_j). \label{HermiteExpansionIndicatorRed}
\end{align}
Let $m(n)$ be the Hermite rank of $(1_{\{G_n(X_j) \leq x \}})_x$. Then we have by the conditions of Theorem \ref{WeakRedPrinArray} that $m^{\ast} \leq m(n)\leq m$ for some $m^{\ast} \leq m < 1/D$. Thus
\begin{align*}
L_{m,n,j}(x) = F_{(n)}(x) + \sum_{q=m^{\ast}}^{m} J_{q,n}(x) /q! H_q(X_j).
\end{align*}
Moreover, define
\begin{align*}
S_n(l;x) = \frac 1 {d_{n,m}} \sum_{j=1}^l \left ( 1_{\{G_n(X_j) \leq x \}} - L_{m,n,j}(x) \right ).
\end{align*}
Finally, let $S_n(k;x,y)= S_n(k;y) -S_n(k;x)$, $L_{m,n,j}(x,y) = L_{m,n,j} (y) - L_{m,n,j}(x)$ and $J_{n,q}(x,y) = J_{n,q}(y)-J_{n,q}(x)$.

We will make use of the chaining technique of \citet{DeTa}. To this end, define
\begin{align*}
\Lambda_n (x) :=  \int_{\{G_n(s)\leq x\}} \left (\sum_{q=0}^m \frac {\lvert H_{q}(s) \rvert} {q!} \right ) \phi(s) \ ds
\end{align*}
and observe that $J_{q,n}(x,y)/q!$ is bounded by $\Lambda_n(x,y) = \Lambda_n(y)- \Lambda_n(x)$, for all $n \in \mathds N$ and all $q=0, \dots , m$. Furthermore, $\Lambda_n$ is monotone, $\Lambda_n(-\infty)=0$ and
\begin{align*}
\Lambda_n(+\infty) =   \int_{\mathds R} \left (\sum_{q=0}^m \frac {\lvert H_{q}(s) \rvert} {q!} \right ) \phi(s) \ ds = C < \infty, \ \ \text{for all } n \in \mathds N.
\end{align*} 
Define partitions, similarly to \citet{DeTa}, but now depending on $n$, by
\begin{align*}
x_i(k) = x_i^{(n)}(k) = \inf \{ x \vert \Lambda_n(x) \geq \Lambda_n(+\infty)i2^{-k} \} \ \ i=0, \cdots, 2^k-1
\end{align*}
for $k=0, \cdots , K$, with the integer $K$ chosen below. Then we have
\begin{align}
\Lambda_n(x_i(k)-) - \Lambda_n(x_{i-1}(k)) \leq \Lambda_n(+\infty)2^{-k}. \label{AbschPartition}
\end{align}
Note that the right hand side of (\ref{AbschPartition}) does not depend on $n$. \\
Based on these partitions we can define chaining points $i_k(x)$ by
\begin{align*}
x_{i_k(x)}(k) \leq x < x_{i_k(x)+1}(k),
\end{align*}
 for each $x$ and each $k \in \{0,1, \dots , K\}$, see \citet{DeTa}.

\begin{Lem}  \label{BeweisStrukturDefinierendesLemma} Define the chaining points as above. Suppose the following two conditions hold:
\begin{enumerate}
\item [(i)] There are constants $\gamma >0$ and $C>0$, not depending on $n$, such that for all $k \leq n$
\begin{align*}
E \lvert S_n(k;x,y) \rvert ^2 \leq C \left ( \frac k n \right ) n^{-\gamma} F_{(n)}(x,y).
\end{align*} 
\item[(ii)] For all $\epsilon >0$ and all $n \in \mathds N$ there is a real number $K=K(n,\epsilon)$, such that for all $\lambda >0$
\begin{align*}
P \left ( \sup_{x \in \mathds R} \left \lvert  \frac 1 {d_{n,m}} \sum_{j=1}^l L_{m,n,j}(x_{i_K(x)}(K),x_{i_K(x)+1}(K)-) \right \rvert > \epsilon \right ) \leq C \left ( \frac l n \right )^{2-m^{\ast}D} n^{\lambda - m^{\ast}D}.
\end{align*}
\end{enumerate}
Then there is a constant $\rho >0$, such that for all $n \in \mathds N$ and all $\epsilon >0$ the following holds:
\begin{align*}
P \left ( \sup_{x} \lvert S_n(l;x) \rvert > \epsilon \right )  \leq C \left ( \frac l n \right) n^{-\gamma} \epsilon^{-2} (K(n,\epsilon)+ 3)^5 + C \left ( \frac l n \right)^{2-m^{\ast}D} n^{\lambda -m^{\ast}D}.
 \end{align*}
\end{Lem}

\begin{proof}
Due to definition of the chaining points each point $x$ is linked to $- \infty$ in detail
\begin{align*}
- \infty = x_{i_0(x)}(0) \leq x_{i_1(x)}(1) \leq \dots \leq x_{i_K(x)}(K) \leq x < x_{i_K(x)+1}(K)
\end{align*}
We have
\begin{align}
S_n(l;x) = & \ \sum_{k=1}^K S_n(l;x_{i_{k-1}(x)}(k-1),x_{i_k(x)}(k)) \ +  \ S_n(l;x_{i_K(x)}(K),x). \label{DazwischenOrange}
\end{align} 
The last summand of the right hand side of (\ref{DazwischenOrange}) can be treated as follows
\begin{equation}
\begin{aligned}
\left \lvert S_n(l;x_{i_K(x)}(K),x) \right \rvert = & \ \frac 1 {d_{n,m}} \left \lvert \sum_{j=1}^l \bigg ( 1_{\{ x_{i_K(x)}(K) < G_n(X_j) \leq x \} } - L_{m,n,j}(x_{i_K(x)}(K),x) \bigg)  \right \rvert \\
\leq & \ \frac 1 {d_{n,m}} \left \lvert \sum_{j=1}^l \bigg ( 1_{\{ x_{i_K(x)}(K) < G_n(X_j) <  x_{i_K(x)+1}(K) \} } - L_{m,n,j}(x_{i_K(x)}(K),x_{i_K(x)+1}(K)-) \bigg) \right \rvert \\
& \ + 2 \frac 1 {d_{n,m}} \left \lvert \sum_{j=1}^l L_{m,n,j}(x_{i_K(x)}(K),x_{i_K(x)+1}(K)-) \right \rvert \\
= & \ \left \lvert S_n(l;x_{i_K(x)}(K),x_{i_K(x)+1}(K)-) \right \rvert  \\
& \ + 2 \frac 1 {d_{n,m}} \left \lvert \sum_{j=1}^l L_{m,n,j}(x_{i_K(x)}(K),x_{i_K(x)+1}(K)-) \right \rvert.
\end{aligned} \label{ChainingNeuGrau-Braun}
\end{equation}
By (\ref{DazwischenOrange}) and (\ref{ChainingNeuGrau-Braun}) we get, using $\sum_{k=1}^{\infty} (k+2)^{-2} < 1/2$,
\begin{align}
\nonumber & \ P \left ( \sup_{x} \lvert S_n(l;x) \rvert > \epsilon \right ) \\
\nonumber \leq & \ P \left ( \sup_{x} \lvert S_n(l;x) \rvert > \epsilon \sum_{k=1}^{K+1}(k+2)^{-2} + \epsilon/2 \right ) \\
\leq & \ \sum_{k=1}^K P \left ( \max_{x} \lvert S_n(l;x_{i_{k-1}(x)}(k-1),x_{i_k(x)}(k)) \rvert > \epsilon /(k+2)^2 \right ) \label{ChainingProb1} \\
 & \ + P \left ( \max_{x} \lvert S_n(l;x_{i_{K}(x)}(K),x_{i_K(x)+1}(K)-) \rvert > \epsilon /(K+3)^2 \right ) \label{ChainingProb2} 
 \\
 & \ + P \left ( 2 d_{n,m}^{-1} \left \lvert \sum_{j \leq l} L_{m,n,j}(x_{i_K(x)}(K),x_{i_K(x)+1}(K)-) \right \rvert > (\epsilon/2) \right ). \label{ChainingProb3}
\end{align}
Further, by condition (i) of Lemma \ref{BeweisStrukturDefinierendesLemma} and the Markov inequality we get
\begin{align}
\nonumber & \ P  \left ( \max_{x} \lvert S_n(l;x_{i_k(x)}(k),x_{i_{k+1}(x)}(k+1)) \rvert > \epsilon /(k+2)^2 \right ) \\
\nonumber \leq & \ \sum_{i=0}^{2^{k+1}-1} P \left ( S_n(l;x_i(k+1),x_{i+1}(k+1))  > \epsilon /(k+2)^2 \right ) \\
\leq & \ C \sum_{i=0}^{2^{k+1}-1}  \left ( \frac l n \right ) n^{-\gamma} \frac {(k+2)^4} {\epsilon^2} F_{(n)} (x_i(k+1), x_{i+1}(k+1)) \label{ChainingMarkovBlau} \\
\nonumber \leq & \ C  \left ( \frac l n \right ) n^{-\gamma} \frac {(k+2)^4} {\epsilon^2}.
\end{align}
The constant $C$ in (\ref{ChainingMarkovBlau}) is the constant of condition (i) in Lemma \ref{BeweisStrukturDefinierendesLemma} and thus independent of $n$. In the next line this $C$ gets multiplied with $\Lambda_n(+\infty)$, which is a constant by itself. Thus the $C$ in the inequality above is a universal constant, not depending on $n$. The same is true for $\gamma$. \\
Using the same arguments we get moreover 
\begin{align*}
P  \left ( \max_{x} \lvert S_n(l;x_{i_K(x)}(K),x_{i_K(x)+1}(K)-) \rvert > \epsilon /(K+3)^2 \right ) \leq C  \left ( \frac l n \right ) n^{-\gamma} \frac {(K+3)^4} {\epsilon^2}.
\end{align*}
Finally we have by condition (ii) of Lemma \ref{BeweisStrukturDefinierendesLemma}
\begin{align*}
P \left ( 2 d_{n,m}^{-1} \left \lvert \sum_{j \leq l} L_{m,n,j}(x_{i_K(x)}(K),x_{i_K(x)+1}(K)-) \right \rvert > (\epsilon/2) \right ) \leq C \left ( \frac l n \right )^{2-m^{\ast}D} n^{\lambda - m^{\ast}D},
\end{align*}
for all $\lambda >0$. Combining the estimates for (\ref{ChainingProb1}), (\ref{ChainingProb2}) and (\ref{ChainingProb3}) we arrive at
\begin{align*}
   P \left ( \sup_{x} \lvert S_n(l;x) \rvert > \epsilon \right ) \leq & \ C \left ( \frac l n \right) n^{-\gamma} \epsilon^{-2}\sum_{k=1}^{K+1} (k+ 2)^4 + C \left ( \frac l n \right)^{2-m^{\ast}D} n^{\lambda -m^{\ast}D} \\
   \leq & \ C \left ( \frac l n \right) n^{-\gamma} \epsilon^{-2} (K+ 3)^5 + C \left ( \frac l n \right)^{2-m^{\ast}D} n^{\lambda -m^{\ast}D}.
  \end{align*}
which finishes the proof.
 \end{proof}

\begin{Lem} \label{MomentUnglRedPrin2}
There exist constants $\gamma$ and $C$, not depending on n, such that for all $k \leq n$
\begin{align*}
E \lvert S_n(k;x,y) \rvert ^2 \leq C \left ( \frac k n \right ) n^{-\gamma} F_{(n)}(x,y).
\end{align*}
\end{Lem}

The proof is very close to the proof of Lemma 3.1 in \citet{DeTa}. However, for further results it is crucial that $C$ and $\gamma$ only depend indirectly on the function $G_n$, namely through the Hermite rank. Thus we give a detailed proof to highlight this fact.

\begin{proof}
First, obtain the Hermite expansion
\begin{align*}
1_{ \{ x < G_n(X_i) \leq y \} } - F_{(n)} (x,y) = \sum_{q=m^{\ast}}^{\infty} \frac {J_{q,n}(x,y)} {q!} H_q(X_i).
\end{align*}
Secondly, we have by orthogonality of the $H_q(X_i)$ and $E H^2_q(X_i)=q!$
\begin{align*}
\sum_{q=m^{\ast}}^{\infty} \frac {J_{q,n}^2(x,y)} {q!} & \ = \sum_{q=m^{\ast}}^{\infty} E \left ( \frac {J_{q,n}(x,y)} {q!}  H_q(X_i) \right )^2 \\
& \ = E \left ( \sum_{q=m^{\ast}}^{\infty} \frac {J_{q,n}(x,y)} {q!} H_q(X_i) \right )^2 \\
& \ = E \left ( 1_{ \{ x < G_n(X_i) \leq y \} } - F_{(n)} (x,y) \right )^2 \\
& \ = F_{(n)}(x,y)(1- F_{(n)}(x,y)) \\
& \ \leq F_{(n)}(x,y).
\end{align*}
This yields
\begin{align*} 
E \left ( d_{n,m} S_n(k;x,y) \right )^2  = & \ \sum_{q =m+1}^{\infty} \frac {J_{q,n}^2(x)} {q!} \frac 1 {q!} \sum_{i,j \leq k} EH_q(X_i)H_q(X_j) \\
\leq & \ F_{(n)}(x,y) \sum_{i,j \leq k} \lvert r(i-j)\rvert^{m+1}.
\end{align*}
Note that the second factor of the product in the last line may depend indirectly on the function $G_n$, because $G_n$ determines $m$, however this is the only influence. For different combinations of $m$ and $D$ the term $\sum_{i,j \leq k} \lvert r(i-j)\rvert^{m+1}$ might have a different asymptotic order. However, in all cases we get (see page 1777 in \citet{DeTa})
\begin{align*}
  \frac 1 {d_{n,m}}\sum_{i,j \leq k} \lvert r(i-j)\rvert^{m+1} & \ \leq C n^{mD-2}L^{-m}(n) k^{1 \vee (2-(m+1))/D} L_1(k) \\
 & \ \leq C   \left ( \frac k n \right )^{1 \vee (2-(m+1)D)} n^{mD -1 \vee (-D)} L_1(k) L^{-m}(n).
  \end{align*}
The result then follows because $L$ and $L_1$ are slowly varying. 
\end{proof}

\begin{Lem} \label{ResttermChaining}
Let $n \in \mathds N$ and $\epsilon >0$. Define the chaining points and $L_{m,n,j}(x)$ as in (\ref{HermiteExpansionIndicatorRed}). Set
\begin{align*}
K = K(n,\epsilon) = \left \lfloor \log_2 \left ( \frac {(m-m^{\ast}+2) \Lambda_n(+\infty)} {\epsilon} n d_{n,m}^{-1} \right ) \right \rfloor+1.
\end{align*}
Then there is a constant $C>0$, such that for all $\lambda >0$
\begin{align*}
P \left ( \sup_{x \in \mathds R} \left \lvert  \frac 1 {d_{n,m}} \sum_{j=1}^l L_{m,n,j}(x_{i_K(x)}(K),x_{i_K(x)+1}(K)-) \right \rvert > \epsilon \right ) \leq C \left ( \frac l n \right )^{2-m^{\ast}D} n^{\lambda - m^{\ast}D}.
\end{align*}
\end{Lem}

\begin{proof}
 By construction of the chaining points we have for $q=0, \dots, m$ and for all $x \in \mathds R$
 \begin{align*}
\sup_{x \in \mathds R} \lvert J_{q,n}(x_{i_K(x)}(K),x_{i_K(x)+1}(K)-) /q! \rvert \leq \Lambda_n(+\infty) 2^{-K}.
\end{align*}
Thus for all $x \in \mathds R$
\begin{align*}
& \ \frac 1 {d_{n,m}} \left \lvert \sum_{j=1}^l L_{m,n,j}(x_{i_K(x)}(K),x_{i_K(x)+1}(K)-) \right \rvert \\
\leq & \ \sum_{q=0}^m \lvert J_{q,n}(x_{i_K(x)}(K),x_{i_K(x)+1}(K)-) /q! \rvert \frac 1 {d_{n,m}} \left \lvert \sum_{j=1}^l H_q(X_j) \right \rvert \\
\leq & \ \Lambda_n(+\infty) 2^{-K} \sum_{q=0}^m  \frac 1 {d_{n,m}} \left \lvert \sum_{j=1}^l H_q(X_j) \right \rvert. 
\end{align*}
By definition of $K$
\begin{align*}
\frac {2^K \epsilon}{(m-m^{\ast}+2) \Lambda_n(+\infty)} \geq \frac {n}{d_{n,m}}.
\end{align*}
Therefore we get by Markov's inequality for $q=m^{\ast}, \dots, m$
\begin{align*}
 P \left ( \Lambda_n(+\infty) 2^{-K}  \frac 1 {d_{n,m}} \left \lvert \sum_{j=1}^l H_q(X_j) \right \rvert > \epsilon/(m-m^{\ast}+2)  \right) \leq & \ P \left ( \left \lvert \sum_{j=1}^l H_q(X_j) \right \rvert > n \right ) \\
\leq & \ C \frac {d_{l,q}^2} {n^2} \\ 
\leq & \ C \frac {l^{2-qD}}{n^2} L^q(l) \\
\leq & \ C \left ( \frac l n \right )^{2-m^{\ast}D} n^{\lambda - m^{\ast}D}.
\end{align*}
For $q=0$ the term is deterministic, thus the probability is $0$.
\end{proof}

\begin{proof} [Proof of Theorem \ref{WeakRedPrinArray}]

The two conditions of Lemma \ref{BeweisStrukturDefinierendesLemma} are satisfied (see Lemma \ref{MomentUnglRedPrin2} and Lemma \ref{ResttermChaining} ) with 
\begin{align*}
K = \left \lfloor \log_2 \left ( \frac {(m-m^{\ast}-2) \Lambda_n(+\infty)} {\epsilon} n d_{n,m}^{-1} \right ) \right \rfloor+1.
\end{align*}
Note that $(K+3)^5 \leq C \epsilon^{-1}n^{\delta}$ for any $\delta >0$, see \citet{DeTa}, page 1781. By this fact and by virtue of Lemma \ref{BeweisStrukturDefinierendesLemma}
\begin{align*}
P \left ( \sup_{x} \lvert S_n(l;x) \rvert > \epsilon \right )  \leq & \ C \left ( \frac l n \right) n^{\delta-\gamma} \epsilon^{-3} + C \left ( \frac l n \right)^{2-m^{\ast}D} n^{\lambda -m^{\ast}D} \\
\leq & \ C n ^{-\rho} \left \{ \left ( \frac l n \right)  \epsilon^{-3} + \left ( \frac l n \right)^{2-m^{\ast}D} \right \},
 \end{align*}
 with $ \rho = \min (\gamma - \delta, m^{\ast}D-\lambda) $. Now choose $\delta < \gamma$, then $\rho > 0$ and we have thus proven a reduction principle in $x$. It remains to verify uniformity in $l$. For $n=2^r$ one gets by the same arguments as in the proof of Theorem 3.1 in \citet{DeTa}
\begin{align*}
P \left ( \max_{l \leq n} \sup_{x} \lvert S_n(l;x) \rvert > \epsilon \right ) \leq Cn^{-\kappa}(1+ \epsilon^{-3})
\end{align*}
for any $0< \epsilon \leq 1$ and universal constants $C$ and $\epsilon$. Next consider arbitrary $n$ and define for $r$ such that $2^{r-1} < n\leq 2^r$
\begin{align*}
S_n^{\ast}(l,x) = \frac 1 {d_{2^r,m}} \sum_{j=1}^l (1_{\{G_n(X_j) \leq x\}} - L_{m,n,j}(x) ) \ \ \ \ \text{for } l \leq 2^r,
\end{align*}
where $\{G_n(X_j) \}_{n \in \mathds N, j \leq 2^r}$ is a (slightly modified) array. One obtains
\begin{align*}
P \left ( \max_{l \leq n} \sup_{x} \lvert S_n^{\ast}(l;x) \rvert > \epsilon \right ) \leq C(2^r)^{-\kappa}(1+ \epsilon^{-3}).
\end{align*}
Hence
\begin{align*}
P \left ( \max_{l \leq n} \sup_{x} \lvert S_n(l;x) \rvert > \epsilon \right ) \leq & \ P \left ( \max_{l \leq n} \sup_{x} \lvert S_n^{\ast}(l;x) \rvert > \epsilon \frac {d_{n,m}}{d_{2^r,m}} \right )  \\
\leq & \ \leq C(2^r)^{-\kappa} \left (1+ \epsilon^{-3} \left (\frac {d_{2^r,m}}{d_{n,m}} \right )^3 \right ) \\
\leq & \ Cn^{-\kappa}(1+ \epsilon^{-3}).
\end{align*}
The last line holds since $d_{2^r,m} / d_{n,m}$ is uniformly bounded away from $0$ and $\infty$. Thus, Theorem \ref{WeakRedPrinArray} is proven. 
\end{proof}

\begin{proof} [Proof of Corollary  \ref{KonvEmpProArray}] 
Using the reduction principle, namely Theorem \ref{WeakRedPrinArray}, it remains to show that
\begin{align}
d_{n,m}^{-1}\sum_{q=m^{\ast}}^m \frac {J_{q,n}(x)} {q!} \sum_{i=1}^{\lfloor nt \rfloor}  H_q(X_i) \label{SummeHermite}
\end{align}
converges to the desired limit processes. Define
\begin{align*}
Z_{n,q}(t) = \frac 1 {d_{n,q}} \sum_{i=1}^{\lfloor nt \rfloor} H_q(X_i),
\end{align*}
and note that because of $1/m > D$ the sequences $\{H_q(X_i)\}_{i \geq 1}$ are long-range dependent for $q=m^{\ast}, \dots, m$. Then we have by Theorem 4 of \citet{BaTa}
\begin{align}
\left ( Z_{n,m^{\ast}}, \dots ,Z_{n,m} \right ) \xrightarrow{ \mathcal D} \left (Z_{m^{\ast}}, \dots, Z_{m} \right ),
\end{align}
where convergence takes place in $(D[0,1])^{m-m^{\ast}+1}$, equipped with the uniform metric. Moreover, $(Z_q(t))_{t \in [0,1]}$ are uncorrelated Hermite processes of order $q$. The functions $h_q$ are elements of $D[-\infty,\infty]$ and therefore they are also bounded, see Remark \ref{BemerkungEigenschaftFunktionen}. Hence we may apply the continuous mapping theorem and conclude that
\begin{align*}
\left \{  \sum_{q=m^{\ast}}^m h_q(x) d_{n,q}^{-1} \sum_{i=1}^{\lfloor nt \rfloor} H_q(X_i)  \right \}_{t,x}
\end{align*}
converges in distribution to
\begin{align*}
\left \{ \sum_{q=m^{\ast}}^m h_q(x) Z_q(t) \right \}_{t,x},
\end{align*}
where convergence takes place in $D([0,1]\times [- \infty, \infty]$, equipped with the supremum norm. The result then follows by the uniform convergence of $d_{n,m}/d_{n,q} J_{q,n}(x)$ towards $q! h_q(x)$, the reduction principle and Slutsky's theorem.
\end{proof}

\subsection{Proof of Theorem \ref{ConvChangeAltA} and Theorem \ref{KorAsVertCPStat}} 


We start by proving a reduction principle for the empirical process in presence of a change point. Consider the array $\{Y_{n,i}\}_{n \in \mathds N, i \leq n}$, defined in section \ref{EmpProcessChangePointAlt}, and let $H_{n,i}(x) = P(Y_{n,i} \leq x)$. Define 
\begin{align*}
S_n^{(\tau)}(t,x) = \frac 1 {d_{n,m}} \sum_{i=1}^{\lfloor nt \rfloor} \left (1_{\{Y_{n,i} \leq x\} } - H_{n,i}(x) - \sum_{q=m^{\ast}}^m \frac {J_{q,n,i}(x)} {q!} H_q(X_i) \right),
\end{align*}
where $J_{q,n,i}(x) = E[ 1_{\{Y_{n,i} \leq x\} } H_q(X_i)]$. Note that $J_{q,n,i}(x) = 0$ if $i \leq \lfloor n \tau \rfloor$ and $q <m$. 

\begin{Lem} \label{RedPrinZusammengesetzt}
Let the conditions of Theorem \ref{ConvChangeAltA} hold. Then there are constants $C \geq 0$ and $\kappa >0$ such that for all $\epsilon >0$
\begin{align*}
P \left (\sup_{t \in [0,1]} \sup_{x \in \mathds R} \lvert S_n^{(\tau)} (t,x) \rvert > \epsilon \right ) \leq Cn^{-\kappa}(1+ \epsilon^{-3}).
\end{align*}
\end{Lem}

\begin{proof}
Define
\begin{align*}
& S_{n,1}(t,x) = \frac 1 {d_{n,m}} \sum_{j=1}^{\lfloor nt \rfloor} \left (1_{\{G(X_j) \leq x\} } - F(x) -  \frac {J_m(x)} {m!} H_m(X_j) \right ) \\
\text{and} \ \ \ \ & S_{n,2} (t,x) = \frac 1 {d_{n,m}} \sum_{j=1}^{\lfloor nt \rfloor} \left (1_{\{G_n(X_j) \leq x\} } - F_{(n)}(x) - \sum_{q=m^{\ast}}^m \frac {J_{q,n}(x)} {q!} H_q(X_j) \right ).
\end{align*}
By Theorem \ref{WeakRedPrinArray} we have
\begin{align}
P \left (\sup_{t \in [0,1]} \sup_{x \in \mathds R} \lvert S_{n,i}(t,x) \rvert > \epsilon \right ) \leq Cn^{-\kappa}(1+ \epsilon^{-3}) \ \ \ \ i=1,2. \label{DoppelteAnwendungTheorem3Gelb}
\end{align}
Next obtain
\begin{align*}
S_n^{(\tau)}(t,x) = \begin{cases}
       S_{n,1}(t,x), & \text{if } t \leq \tau, \\
       S_{n,2}(t,x) + S_{n,1}(\tau,x) -S_{n,2}(\tau,x), & \text{if } t> \tau.
       \end{cases}
\end{align*}
Therefore, we get, using (\ref{DoppelteAnwendungTheorem3Gelb}) several times,
\begin{align*}
P \left (\sup_{t \in [0,1]} \sup_{x \in \mathds R} \lvert S_n^{(\tau)} (t,x) \rvert > \epsilon \right ) \leq & \ 2 P \left (\sup_{t \in [0,1]} \sup_{x \in \mathds R} \lvert S_{n,1}(t,x) \rvert > \epsilon/4 \right ) \\
& \ + 2 P \left (\sup_{t \in [0,1]} \sup_{x \in \mathds R} \lvert S_{n,2}(t,x) \rvert > \epsilon/4 \right ) \\
\leq & \ 4 Cn^{-\kappa}(1+ \epsilon^{-3}),
\end{align*}

for all $n \in \mathds N$ and all $\epsilon >0$. 
\end{proof}

\begin{Lem} \label{KonvergenzHermiteKoeffizientFunktionen}
Let Assumption A hold. Then for all $q \leq m$
\begin{align}
\sup_{x \in \mathds R} d_{n,m^{\ast}}/d_{n,m} \lvert J_{q,n}(x) -J_q(x) \rvert \rightarrow  0, \label{KonvHermiteKoeffizientenRot}
\end{align}
as $n \to \infty$.
\end{Lem}

\begin{proof}
Using H\"{o}lder's inequality, one has for any $p \in \mathds N$
\begin{align*}
\lvert J_{q,n} (x)- J_q(x) \rvert = & \ \lvert E \left ( ( 1_{ \{ G_n(X_i) \leq x \} } -  1_{ \{ G(X_i) \leq x \} } ) H_q(X_i) \right) \rvert  \\
\leq & \ \left ( E \lvert 1_{ \{ G_n(X_i) \leq x \} } -  1_{ \{ G(X_i) \leq x \} } \rvert ^{(p+1)/p} \right) ^{p/(p+1)} \lVert H_q(X_i) \rVert_{L^{p+1}} \\
\leq & \ C( E \lvert 1_{ \{ G_n(X_i) \leq x \} } -  1_{ \{ G(X_i) \leq x \} } \rvert)^{p/(p+1)}
\end{align*}
Now obtain
\begin{align*}
& \ E \lvert 1_{ \{ G_n(X_i) \leq x \} } -  1_{ \{ G(X_i) \leq x \} } \rvert \\
= & \ P \left ( \{G_n(X_1) \leq x, G(X_1) > x \} \cup \{G_n(X_1) > x, G(X_1) \leq x \} \right ) \\
= & \ 1 - P \left ( \{G_n(X_1) \leq x, G(X_1) \leq x \} \right) - P \left ( \{G_n(X_1) > x, G(X_1) > x \} \right) \\
= & \ P( \min \{G_n(X_1),G(X_1)\} \leq x) - P( \max \{G_n(X_1),G(X_1)\} \leq x) \\
= & \ o(n^{(m^{\ast}-m)D(1+ \delta)/2}),
\end{align*}
for some $\delta >0$. The last line holds uniformly due to Assumptions A2. Finally,
\begin{align*}
& \ d_{n,m^{\ast}}/d_{n,m} \lvert J_{q,n}(x) -J_q(x) \rvert \\
\leq & \ C n^{(m-m^{\ast})D/2} L^{(m^{\ast}-m)/2}(n) ( E \lvert 1_{ \{ G_n(X_i) \leq x \} } -  1_{ \{ G(X_i) \leq x \} } \rvert)^{p/(p+1)} \\
\leq & \ C \left ( n^{(m-m^{\ast})D(p+1)/p/2} E \lvert 1_{ \{ G_n(X_i) \leq x \} } -  1_{ \{ G(X_i) \leq x \} } \rvert \right ) ^{p/(p+1)} \\
=  & \ C \left ( n^{(m-m^{\ast})D(1+1/p)/2} E \lvert 1_{ \{ G_n(X_i) \leq x \} } -  1_{ \{ G(X_i) \leq x \} } \rvert \right ) ^{p/(p+1)}.
\end{align*}
Choosing $p> 1/\delta$, this implies (\ref{KonvHermiteKoeffizientenRot}).
\end{proof}

\begin{proof} [Proof of Theorem  \ref{ConvChangeAltA}]
By definition of the functions $J_{q,n,i}$ we get
\begin{align*}
& \ \frac 1 {d_{n,m}} \sum_{i=1} ^{\lfloor n t \rfloor} \sum_{q=m^{\ast}}^m \frac {J_{q,n,i}(x)} {q!} H_q(X_i) \\
= & \ \frac 1 {d_{n,m}} \frac {J_m(x)} {m!} \sum_{i=1} ^{\lfloor n t \rfloor} H_m(X_i) \\
& \ + 1_{\{t > \tau\}} \sum_{q=m^{\ast}}^{m-1} \frac {d_{n,q}} {d_{n,m}} \frac {J_{q,n}(x)} {q!} \frac 1 {d_{n,q}} \sum_{i= \lfloor n \tau \rfloor +1} ^{\lfloor n t \rfloor} H_q(X_i)\\
 & \ + 1_{\{t > \tau\}} \frac {J_{m,n}(x) - J_m(x) } {m!} \frac 1 {d_{n,m}} \sum_{i= \lfloor n \tau \rfloor +1} ^{\lfloor n t \rfloor} H_m(X_i).
\end{align*}
The second and the third summands are negligible due to the uniform convergence of the functions $J_{q,n}$ (see Lemma \ref{KonvergenzHermiteKoeffizientFunktionen})). The first summand  converges in distribution towards 
\begin{align*}
\frac{J_m(x)}{m!} Z_m(t),
\end{align*}
see \citet{DeTa}. Together with Lemma \ref{RedPrinZusammengesetzt} this finishes the proof.
\end{proof}

\begin{proof} [Proof of Theorem \ref{KorAsVertCPStat}]
We give the proof for a sequence of local alternatives. The asymptotic behavior under the hypothesis then is an immediate consequence. Obtain the following decomposition of the empirical bridge-process
\begin{align}
& \ \frac 1 {d_{n,m}} \left ( \sum_{i=1}^{\lfloor n t \rfloor} 1_{\{Y_{n,i} \leq x \}} - \frac {\lfloor n t \rfloor} {n} \sum_{i=1}^n 1_{\{Y_{n,i} \leq x \}} \right ) \label{EmpiricalBridgeProzess} \\
\nonumber = & \  \frac 1 {d_{n,m}} \left (\sum_{i=1}^{\lfloor n t \rfloor} \left (1_{\{Y_{n,i} \leq x \}} - H_{n,i}(x) \right )- t \sum_{i=1}^n \left (1_{\{Y_{n,i} \leq x \}} - H_{n,i}(x) \right ) \right ) \\
\nonumber & \ + \left ( t - \frac {\lfloor n t \rfloor} {n} \right) \frac 1 {d_{n,m}} \sum_{i=1}^n \left (1_{\{Y_{n,i} \leq x \}} - H_{n,i}(x) \right ) \\
\nonumber  & \ + \frac n {d_{n,m}} \psi_{n,\tau} (t) \left ( F(x) - F_{(n)}(x) \right ),
\end{align}
where
\vspace{-10pt}
\begin{align*}
\psi_{n,\tau} (t) = \begin{cases} 
                             \frac {\lfloor nt \rfloor} n  \left ( 1- \frac {\lfloor n \tau \rfloor} n \right ),  & \text{if }t \leq   \tau  ,\\
                         \frac {\lfloor n \tau \rfloor} n  \left ( 1- \frac {\lfloor n t \rfloor} n \right ), & \text{if } t > \tau.
                                                      \end{cases}
\end{align*}
By uniform convergence of $n/d_{n,m}(F(x) - F_{(n)} (x))$ and $\psi_{n,\tau}(t)$ towards $g(x)$ and $\psi_{\tau}(t)$, respectively, Theorem \ref{ConvChangeAltA} and the continuous mapping theorem, one gets that (\ref{EmpiricalBridgeProzess}) converges weakly towards
\begin{align*}
J_m(x)/(m!) \left ( Z_m(t) - t Z_m(t) \right) + \psi_{\tau} (t) g(x).
\end{align*}
The convergence of the Kolmogorov-Smirnov type statistic then follows from continuity of the application of the supremum norm. The Cram\'{e}r-von Mises statistic $S_n$ can be written $S_n = \sup_{t \in [0,1]} M_n(t)$, where
\begin{align}
\nonumber M_n(t)  = & \ d_{n,m}^{-2}  \int_{\mathds R} \left ( \sum_{i=1}^{ \lfloor nt \rfloor } 1_{ \{ Y_{n,i} \leq x \} } - \frac{\lfloor nt \rfloor}{n} \sum_{i=1}^n 1_{ \{ Y_{n,i} \leq x \} } \right ) ^2 \ d \hat F_n(x)\displaybreak \\
= & \ d_{n,m}^{-2}  \int_{\mathds R} \left ( \sum_{i=1}^{ \lfloor nt \rfloor } 1_{ \{ Y_{n,i} \leq x \} } - \frac{\lfloor nt \rfloor}{n} \sum_{i=1}^n 1_{ \{ Y_{n,i} \leq x \} } \right ) ^2 \ d F(x)  \label{ZerCvMRot} \\
 & \ + d_{n,m}^{-2}   \int_{\mathds R} \left ( \sum_{i=1}^{ \lfloor nt \rfloor } 1_{ \{ Y_{n,i} \leq x \} } - \frac{\lfloor nt \rfloor}{n} \sum_{i=1}^n 1_{ \{ Y_{n,i} \leq x \} } \right ) ^2 \ d (\hat F_n(x)-F(x)) .\label{ZerCvMBlau}
\end{align}
Due to the convergence of (\ref{EmpiricalBridgeProzess}) and the continuous mapping theorem, (\ref{ZerCvMRot}) converges to the desired limit process. Thus, it remains to show that (\ref{ZerCvMBlau}) is negligible. Therefore, obtain 
\begin{align}
\nonumber & \ d_{n,m}^{-2} \int_{\mathds R} \left ( \sum_{i=1}^{ \lfloor nt \rfloor } 1_{ \{ Y_{n,i} \leq x \} } - \frac{\lfloor nt \rfloor}{n} \sum_{i=1}^n 1_{ \{ Y_{n,i} \leq x \} } \right ) ^2 \ d (\hat F_n(x)-F(x)) \\
= & \ \int_{\mathds R} \left ( J_m(x)/(m!) \tilde Z_m(t) - \psi_{\tau}(t)g(x) \right )^2  \ d (\hat F_n(x)-F(x)) \label{ZerCvM2Orange} \\
\nonumber & \ + \int_{\mathds R} \Bigg \{ d_{n,m}^{-2} \left ( \sum_{i=1}^{ \lfloor nt \rfloor } 1_{ \{ Y_{n,i} \leq x \} } - \frac{\lfloor nt \rfloor}{n} \sum_{i=1}^n 1_{ \{ Y_{n,i} \leq x \} } \right ) ^2 \\
& \ \ \ \ \ \ \ \ - \left ( J_m(x)/(m!) \tilde Z_m(t) - \psi_{\tau}(t)g(x) \right )^2 \Bigg \}  \ d (\hat F_n(x)-F(x)). \label{ZerCvM2Gelb}
\end{align}
Using the Skorohod-Dudley-Wichura representation theorem (whose conditions are satisfied because $J_m(x)/(m!) \tilde Z_m(t) - \psi_{\tau}(t)g(x)$ lays almost surely in $C([0,1]\times [-\infty, \infty])$), one can assume without loss of generality that
\begin{align*}
d_{n,m}^{-2} \left ( \sum_{i=1}^{ \lfloor nt \rfloor } 1_{ \{ Y_{n,i} \leq x \} } - \frac{\lfloor nt \rfloor}{n} \sum_{i=1}^n 1_{ \{ Y_{n,i} \leq x \} } \right ) ^2  - \Bigg ( J_m(x)/(m!) \tilde Z_m(t) - \psi_{\tau}(t)g(x) \Bigg )^2
\end{align*}
converges almost surely to $0$, uniformly in $x$ and $t$. Thus, (\ref{ZerCvM2Gelb}) converges to $0$, uniformly in $t$. Next consider (\ref{ZerCvM2Orange})
\begin{align*}
& \ \int_{\mathds R} \left ( J_m(x)/(m!) \tilde Z_m(t) - \psi_{\tau}(t)g(x) \right )^2  \ d (\hat F_n(x)-F(x)) \\
=  & \ (\tilde Z_m(t))^2/(m!)^2 \int_{\mathds R} J_m^2(x)  \ d (\hat F_n(x)-F(x)) \\
 & \ - 2 \tilde Z_m(t) \psi_{\tau}(t) /(m!) \int_{\mathds R} J_m(x) g(x)  \ d (\hat F_n(x)-F(x)) \\
 & \ + \psi_{\tau}^2(t) \int_{\mathds R} g^2(x)  \ d (\hat F_n(x)-F(x)) \\
= & \ I_n -II_n+III_n.
\end{align*}
As a consequence of Theorem \ref{ConvChangeAltA} and $F_{(n)}(x) \rightarrow F(x)$ one gets a weak Glivenko-Cantelli type convergence, in detail 
\begin{align*}
\sup_{x \in \mathds R} \lvert \hat F_n(x) - F(x) \rvert \leq \sup_{x \in \mathds R} \left \lvert \hat F_n(x) - \sum_{i=1}^n \frac {H_{n,i}(x)} n \right \rvert +  \frac {n - \lfloor n \tau \rfloor} n \sup_{x \in \mathds R} \rvert F_{(n)}(x) - F(x) \lvert \xrightarrow{P} 0.
\end{align*}
Moreover, obtain that $J_m(x)$ is of bounded variation (this was also noted in \citet{DeTa}). To see this, let $[a,b]$ be an arbitrary interval and $\{x_i\}_{i=0}^n$ a partition of this interval. Then
\begin{align*}
\sum_{i=0}^{n-1} \lvert J(x_{i+1})-J(x_i)\rvert  = & \ \sum_{i=0}^{n-1} \lvert E [1_{\{ x_i < G(X_1) \leq x_{i+1} \}}H_m(X_1) ]\rvert \\
\leq & \ \sum_{i=0}^{n-1} E [1_{\{ x_i < G(X_1) \leq x_{i+1} \}} \lvert H_m(X_1) \rvert ] \\
= & \ E \left [ \sum_{i=0}^{n-1} 1_{\{ x_i < G(X_1) \leq x_{i+1} \}} \lvert H_m(X_1) \rvert \right ] \\
= & \ E \left [ 1_{\{ G(X_1) \in [a,b] \}} \lvert H_m(X_1) \rvert \right ] \\ 
\leq & \ E \lvert H_m(X_1) \rvert.
\end{align*}
By the boundedness of $J_m$, $J_m^2$ is also of bounded variation and thus integration by parts, together with the weak Glivenko-Cantelli-type reulst, yields
\begin{align*}
I_n = - (\tilde Z_m(t))^2/(m!)^2 \int_{\mathds R}  (\hat F_n(x)-F(x)) \ d J_m^2(x) \xrightarrow{P} 0.
\end{align*}
By definition, the function $g(x)$ is bounded and of bounded variation. Hence the same is true for $g^2(x)$ and by the same arguments as above one gets $III_n = o_P(1)$. Finally, $II_n =o_P(1)$, which can be seen using H\"{o}lders's inequality. This finishes the proof.
\end{proof}

\begin{Bem} 
Note that our proof of the weak convergence of the Cram\'{e}r-von Mises statistic would not work for short-range dependent time series. The reason is the completely different limit behavior of the sequential empirical process. Instead of the semi-degenerate process $J_m(x)Z_m(t)$ one gets a Gaussian process $K(t,x)$. While $J_m$ is of bounded variation, this is not the case for sample paths of $K$. Hence  $\int_{\mathds R} K(t,x) \ d (F_n(x) -F(x))$ cannot be treated simultaneously to $\int_{\mathds R} J_m(x) \ d (F_n(x) -F(x))$.
\end{Bem}

\section*{Acknowledgements} The work has been supported by the Collaborative Research Center ``Statistical modeling of nonlinear dynamic processes'' (SFB 823) of the German Research Foundation (DFG). The author would like to thank Herold Dehling for insightful remarks and fruitful discussions.

\bibliography{Lit}

\end{document}